\documentclass[a4paper]{amsart}

\usepackage[T1]{fontenc}
\usepackage[utf8x]{inputenc}
\usepackage[english]{babel}
\usepackage{yfonts}
\usepackage{dsfont}
\usepackage{amscd,amssymb,amsmath,amsthm,amsfonts}
\usepackage{mathtools}
\usepackage{accents}
\usepackage{tensor}
\usepackage{graphicx}
\usepackage{tikz}
\usetikzlibrary{calc,matrix,arrows,decorations.pathmorphing}
\usepackage{calc}
\usepackage[cal=boondox,scr=boondoxo]{mathalfa}
\usepackage{enumitem}
\usepackage{marginnote}
\usepackage{scalerel}[2016/12/29]

\usepackage{hyperref}
\hypersetup{colorlinks=true,pageanchor=false,linkcolor=blue,citecolor=red}
\usepackage[all]{hypcap}

\newtheorem{theorem}{Theorem}[section]
\newtheorem{corollary}[theorem]{Corollary}
\newtheorem{proposition}[theorem]{Proposition}
\newtheorem{lemma}[theorem]{Lemma}

\theoremstyle{definition}
\newtheorem{definition}[theorem]{Definition}

\theoremstyle{remark}
\newtheorem{remark}[theorem]{Remark}

\newcommand{\N}{\mathbb{N}}
\newcommand{\Z}{\mathbb{Z}}

\newcommand{\R}{\mathbb{R}}
\newcommand{\C}{\mathbb{C}}

\newcommand{\rmd}{\mathrm{d}}

\newcommand{\rmt}{\mathrm{t}}

\newcommand{\rmH}{\mathrm{H}}
\newcommand{\rmI}{\mathrm{I}}

\newcommand{\rmN}{\mathrm{N}}

\newcommand{\rmV}{\mathrm{V}}

\newcommand{\bbD}{\mathbb{D}}

\newcommand{\bbM}{\mathbb{M}}
\newcommand{\bbN}{\mathbb{N}}

\newcommand{\bbS}{\mathbb{S}}

\newcommand{\bbV}{\mathbb{V}}

\DeclareRobustCommand{\bbSigma}{\mathbin{\text{\includegraphics[height=\heightof{$\mathbf{\Sigma}$}]{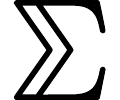}}}}

\newcommand{\bfP}{\mathbf{P}}

\newcommand{\bfepsilon}{\boldsymbol{\varepsilon}}

\newcommand{\bfLambda}{\mathbf{\Lambda}}

\newcommand{\calC}{\mathcal{C}}
\newcommand{\calD}{\mathcal{D}}

\newcommand{\calH}{\mathcal{H}}

\newcommand{\calL}{\mathcal{L}}

\newcommand{\calR}{\mathcal{R}}

\newcommand{\calU}{\mathcal{U}}
\newcommand{\calV}{\mathcal{V}}

\newcommand{\frakg}{\mathfrak{g}}
\newcommand{\frakh}{\mathfrak{h}}
\newcommand{\frakn}{\mathfrak{n}}

\renewcommand{\epsilon}{\varepsilon}
\renewcommand{\theta}{\vartheta}
\renewcommand{\phi}{\varphi}
\renewcommand{\Gamma}{\varGamma}
\renewcommand{\Sigma}{\varSigma}

\newcommand{\id}{\mathrm{id}}

\newcommand{\tr}{\mathrm{tr}}
\newcommand{\ptr}{\mathrm{ptr}}
\newcommand{\im}{\operatorname{im}}

\newcommand{\lev}{\smash{\stackrel{\leftarrow}{\mathrm{ev}}}}
\newcommand{\lcoev}{\smash{\stackrel{\longleftarrow}{\mathrm{coev}}}}
\newcommand{\rev}{\smash{\stackrel{\rightarrow}{\mathrm{ev}}}}
\newcommand{\rcoev}{\smash{\stackrel{\longrightarrow}{\mathrm{coev}}}}
\DeclareMathOperator{\Exists}{\exists}
\DeclareMathOperator{\Forall}{\forall}

\DeclareMathOperator{\disjun}{\sqcup}

\newcommand{\kk}{\Bbbk}

\newcommand{\bp}[1]{{\left(#1\right)}}

\newcommand{\cat}{\mathcal{C}}
\newcommand{\brk}[1]{{{\left\langle{#1}\right\rangle}}}

\renewcommand{\leq}{\leqslant} 
\renewcommand{\geq}{\geqslant}
\newcommand{\sltwo}{\mathfrak{sl}_2}
\newcommand{\GL}{\mathrm{GL}}
\newcommand{\Hom}{\mathrm{Hom}}
\newcommand{\bbHom}{\mathbb{H}\mathrm{om}}
\newcommand{\End}{\mathrm{End}}

\newcommand{\Vect}{\mathrm{Vect}}

\newcommand{\op}{\mathrm{op}}

\newcommand\restr[3][0]{{\raisebox{-#1pt}{$\left. \raisebox{#1pt}{$#2$} \vphantom{\big|} \right|_{#3}$}}}

\DeclareRobustCommand{\one}{\mathbin{\text{\includegraphics[height=\heightof{$\mathbf{1}$}]{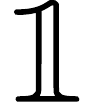}}}}

\newcommand{\PGr}{Z}

\newcommand{\adCob}{\check{\mathrm{C}}\mathrm{ob}}

\newcommand{\Mat}{\mathrm{Mat}}

\newcommand{\sqbinom}[2]{\left[ \begin{matrix} #1 \\ #2 \end{matrix} \right]}
\newcommand{\Proj}{\mathrm{Proj}}

\newcommand{\adrmS}{\check{\mathrm{S}}}

\makeatletter
\DeclareRobustCommand{\cev}[1]{
  \mathpalette\do@cev{#1}
}
\newcommand{\do@cev}[2]{
  \fix@cev{#1}{+}
  \reflectbox{$\m@th#1\vec{\reflectbox{$\fix@cev{#1}{-}\m@th#1#2\fix@cev{#1}{+}$}}$}
  \fix@cev{#1}{-}
}
\newcommand{\fix@cev}[2]{
  \ifx#1\displaystyle
    \mkern#23mu
  \else
    \ifx#1\textstyle
      \mkern#23mu
    \else
      \ifx#1\scriptstyle
        \mkern#22mu
      \else
        \mkern#22mu
      \fi
    \fi
  \fi
}
\makeatother
\makeatletter
\newcommand{\subalign}[1]{
  \vcenter{
    \Let@ \restore@math@cr \default@tag
    \baselineskip\fontdimen10 \scriptfont\tw@
    \advance\baselineskip\fontdimen12 \scriptfont\tw@
    \lineskip\thr@@\fontdimen8 \scriptfont\thr@@
    \lineskiplimit\lineskip
    \ialign{\hfil$\m@th\scriptstyle##$&$\m@th\scriptstyle{}##$\crcr
      #1\crcr
    }
  }
}
\makeatother

\def\clap#1{\hbox to 0pt{\hss#1\hss}}

\def\mathclap{\mathpalette\mathclapinternal}

\def\mathclapinternal#1#2{
\clap{$\mathsurround=0pt#1{#2}$}}

\begin{document}

\raggedbottom

\title[Non-Semisimple Quantum Invariants and TQFTs]{Non-Semisimple Quantum Invariants and TQFTs from Small and Unrolled Quantum Groups}

\author[M. De Renzi]{Marco De Renzi} \address{Department of Mathematics, Faculty of Science and Engineering, Waseda University, 3-4-1 \={O}kubo, Shinjuku-ku, Tokyo, 169-8555, Japan} \email{m.derenzi@kurenai.waseda.jp}

\author[N. Geer]{Nathan Geer}
\address{Mathematics \& Statistics\\
  Utah State University \\
  Logan, Utah 84322, USA} \email{nathan.geer@gmail.com}

\author[B. Patureau-Mirand]{Bertrand Patureau-Mirand}
\address{Univ. Bretagne - Sud, UMR 6205, LMBA, F-56000 Vannes, France}
\email{bertrand.patureau@univ-ubs.fr}

\begin{abstract}
 We show that unrolled quantum groups at odd roots of unity give rise to relative modular categories. These are the main building blocks for the construction of 1+1+1-TQFTs extending CGP invariants, which are non-semisimple quantum invariants of closed 3-manifolds decorated with ribbon graphs and cohomology classes. When we consider the zero cohomology class, these quantum invariants are shown to coincide with the renormalized Hennings invariants coming from the corresponding small quantum groups.
\end{abstract}

\maketitle
\setcounter{tocdepth}{3}

\date{\today}

The goal of this paper is two-fold: first of all, we provide a new family of concrete examples of relative modular categories. These are ribbon categories which can be used as fundamental bricks for the construction of non-semisimple quantum invariants of closed manifolds in dimension 3 \cite{CGP14} and Extended Topological Quantum Field Theories (ETQFTs) in dimension 1+1+1 \cite{D17}. They are modeled on categories of finite-dimensional weight representations of the unrolled quantum group $U_q^H \sltwo$ at even roots of unity $q$, which were used in \cite{BCGP16} to build TQFTs in dimension 2+1. They should be thought of as a non-semisimple analogue to standard modular categories, although differences with respect to their semisimple counterparts are many: first of all, a relative modular category $\calC$ comes equipped with a \textit{structure group} $G$ that provides a grading on its objects; secondly, it enjoys finiteness properties only up to the action of a \textit{periodicity group} $Z$ of transparent objects; more importantly, it is only generically semisimple, with non-semisimple part confined to a \textit{critical set} $X \subset G$ whose complement is dense in $G$. When $G = Z = \{ 0 \}$ and $X = \varnothing$ the definition reduces to the standard one of \cite{T94}, see Section 1.7 of \cite{D17}. In Theorem \ref{T:relative_modularity} we prove that categories $\calC^H$ of finite-dimensional weight representations of unrolled quantum groups $U_q^H \frakg$ associated with arbitrary simple complex Lie algebras $\frakg$ at odd roots of unity $q$ are relative modular. These categories were already known to induce topological invariants $\rmN_{\calC^H}$ of certain decorated closed 3-manifolds $(M,T,\omega)$, where $T \subset M$ is a $\calC^H$-colored ribbon graph, where $\omega \in H^1(M \smallsetminus T;G)$ is a cohomology class, and where the triple $(M,T,\omega)$ is subject to a crucial \textit{admissibility} condition. These so-called \textit{Costantino-Geer-Patureau} (\textit{CGP}) quantum invariants are quite rich, but their definition involves certain technical aspects, such as the notion of \textit{computable} surgery presentation introduced in \cite{CGP14}. The results of this paper imply these invariants can be extended to graded 1+1+1-TQFTs for all simple complex Lie algebras $\frakg$. In the case of $\sltwo$, these invariants contain the Akutsu-Deguchi-Ohtsuki invariants of colored links and the abelian Reidemeister torsion of closed 3-manifolds, and they were already known to extend to graded 1+1+1-TQFTs.

The second main result of this paper is a Hennings-type formula for these CGP quantum invariants. More precisely, every simple complex Lie algebra $\frakg$ also determines a corresponding small quantum group $\bar{U}_q \frakg$ for every odd root of unity $q$ \cite{L90,L93}. These finite-dimensional factorizable quotients have been studied a lot in literature \cite{L95,LN15,LO17}, and they induce renormalized Hennings TQFTs in dimension 2+1, see \cite{DGP18}. In particular, their categories of finite-dimensional representations $\bar{\calC}$ yield quantum invariants $\rmH'_{\bar{\calC}}$ of certain admissible closed 3-manifolds $(M,T)$, where $T \subset M$ is a $\bar{\calC}$-colored bichrome graph, which is a very mild generalization of standard ribbon graphs obtained by specifying special components which correspond to surgery presentations. Then, we prove in Theorem \ref{T:main_result} that, for a fixed simple complex Lie algebra $\frakg$ at an odd root of unity $q$, the CGP invariant $\rmN_{\calC^H}$ of decorated closed 3-manifolds with zero cohomology classes coincides with the corresponding renormalized Hennings invariant $\rmH'_{\bar{\calC}}$. This result builds a bridge between the two theories, and, in this setting, it gives us a way of computing CGP quantum invariants which bypasses computable surgery presentations.

\subsection*{Acknowledgments}
We would like to thank the referee for their extremely careful review of our paper. Their deep understanding of our results and their detailed comments helped us improve the paper. 
 NG was partially supported by NSF grants 
DMS-1308196 and DMS-1452093. 

\section{Overview of non-semisimple constructions}

In this section, we quickly review the two main constructions this paper deals with. References are provided by \cite{CGP14,BCGP16,D17} for the CGP theory, and by \cite{DGP18} for the renormalized Hennings one. All the manifolds we consider are always assumed to be oriented.

\subsection{Ribbon categories and m-traces}

First, let us fix our notation 
and conventions for categorical structures. Following \cite{EGNO15}, a \textit{ribbon category} is a braided rigid monoidal category $\calC$ equipped with a natural transformation $\theta : \id_\calC \Rightarrow \id_\calC$, called the \textit{twist}, satisfying
\[
 \theta_{V \otimes V'} = (\theta_V \otimes \theta_{V'}) \circ c_{V',V} \circ c_{V,V'}, \qquad
 (\theta_V)^* = \theta_{V^*}
\]
for all $V,V' \in \calC$, where $c_{V,V'} : V \otimes V' \to V' \otimes V$ denotes the braiding of $\calC$. For every $V \in \calC$ we denote with
\begin{align*}
 \lev_V &: V^* \otimes V \to \one, &
 \lcoev_V &: \one \to V \otimes V^*, \\
 \rev_V &: V \otimes V^* \to \one, &
 \rcoev_V &: \one \to V^* \otimes V
\end{align*}
its left and right duality morphisms, 
%
%
and for every $f \in \End_\calC(V)$ we denote with $\tr_\calC(f)$ its categorical trace. As shown in \cite{T94}, every ribbon category $\calC$ induces a ribbon functor $F_\calC : \calR_\calC \to \calC$ called the \textit{Reshetikhin-Turaev functor}, where $\calR_{\calC}$ denotes the category of $\calC$-colored ribbon graphs.

An ideal of $\calC$ is a full subcategory of $\calC$ whose class of objects is closed under retracts and absorbent under tensor products with arbitrary objects of $\calC$.
We denote with $\Proj(\calC)$ the ideal of projective objects of $\calC$. We define the \textit{partial trace} of an endomorphism $f \in \End_{\calC}(V \otimes V')$ to be the endomorphism $\ptr(f) \in \End_{\calC}(V)$ given by
\[
 \ptr(f) := (\id_V \otimes \rev_{V'}) \circ (f \otimes \id_{V'^*}) \circ (\id_V \otimes \lcoev_{V'}).
\]
Then, following \cite{GKP11,GKP18}, if $\calC$ is a ribbon linear category over a field $\Bbbk$, an \textit{m-trace $\rmt$ on $\Proj(\calC)$} is a family of linear maps
$\{ \rmt_V : \End_{\calC}(V) \rightarrow \Bbbk \mid V \in \Proj(\calC) \}$ satisfying:
\begin{enumerate}
 \item \textit{Cyclicity}: $\rmt_{V}(f' \circ f) = \rmt_{V'}(f \circ f')$ for all objects $V,V' \in \Proj(\calC)$ and for all morphisms $f \in \Hom_{\calC}(V,V')$ and $f' \in \Hom_{\calC}(V',V)$;
 \item \textit{Partial trace}: $\rmt_{V \otimes V'} (f) = \rmt_V(\ptr(f))$ for all objects $V \in \Proj(\calC)$ and $V' \in \calC$ and for every morphism $f \in \End_{\calC}(V \otimes V')$.
\end{enumerate}
For every $V \in \Proj(\calC)$ we denote with $\rmd(V) := \rmt_V(\id_V)$ its \textit{modified dimension}. We say an m-trace $\rmt$ on $\Proj(\calC)$ is \textit{non-degenerate} if, for every $V \in \Proj(\calC)$ and $V' \in \calC$, the bilinear pairing $\rmt_V( \cdot \circ \cdot ) : \Hom_\calC(V',V) \times \Hom_\calC(V,V') \to \Bbbk$ is non-degenerate.

Finally, let us fix some terminology which will be extensively used throughout the paper: every time we have a ribbon linear category $\calC$ over a field $\Bbbk$, we have an associated notion of \textit{skein equivalence} between formal linear combinations of $\calC$-colored ribbon graphs. Indeed, if $(\underline{\varepsilon},\underline{V})$ and $(\underline{\varepsilon'},\underline{V'})$ are objects of $\calR_{\calC}$, if $\alpha_1,\ldots,\alpha_m,\alpha'_1,\ldots,\alpha'_{m'}$ are scalar coefficients in $\Bbbk$, and if $T_1,\ldots,T_m,T'_1,\ldots,T'_{m'}$ are morphisms of $\calR_{\calC}$ from $(\underline{\varepsilon},\underline{V})$ and $(\underline{\varepsilon'},\underline{V'})$, then we say two formal linear combinations $\sum_{i=1}^m \alpha_i \cdot T_i$ and $\sum_{i'=1}^{m'} \alpha'_{i'} \cdot T'_{i'}$ are \textit{skein equivalent}, and we write 
\[
 \sum_{i=1}^m \alpha_i \cdot T_i \doteq \sum_{i'=1}^{m'} \alpha'_{i'} \cdot T'_{i'},
\]
if we have the equality $\sum_{i=1}^m \alpha_i \cdot F_{\calC}(T_i) = \sum_{i'=1}^{m'} \alpha'_{i'} \cdot F_{\calC}(T'_{i'})$ under the Reshetikhin-Turaev functor $F_{\calC} : \calR_{\calC} \rightarrow \calC$.

\subsection{3-Manifold invariants from non-degenerate relative pre-modular categories}\label{Subs:CGP_invariants}

Let us start by recalling the definition of relative pre-modular categories, which, as we mentioned earlier, are ribbon linear categories carrying additional structures. First of all, if $G$ is an abelian group, a \textit{compatible $G$-structure} on a rigid
monoidal category $\calC$ is an equivalence of linear categories
$\calC \cong \bigoplus_{g \in G} \calC_g$ for a family
$\{ \calC_g \mid g \in G \}$ of full subcategories of $\calC$
satisfying the following conditions: If $V \in \calC_g$, then
$V^* \in \calC_{- g}$; If $V \in \calC_g$ and $V' \in \calC_{g'}$,
then $V \otimes V' \in \calC_{g + g'}$; If $V \in \calC_g$ and
$V' \in \calC_{g'}$ with $g \neq g'$, then $\Hom_{\calC}(V,V') =
0$. Remark that, if $\calC$ is ribbon, then $\calC_0$ is also ribbon.
Next, if $Z$ is an abelian group, a \textit{free realization of $\PGr$ in a ribbon category $\calC$} is a monoidal functor $\sigma : Z \rightarrow \calC$, where $Z$ also denotes the discrete category over $Z$ with tensor product given by the group operation $+$, satisfying $\vartheta_{\sigma(k)} = \id_{\sigma(k)}$ for every $k \in Z$, and inducing a free action on isomorphism classes of simple objects of $\calC$ by tensor product with $\sigma(k)$. 
Next, we say a subset $X$ of $G$ is \textit{symmetric} if $X = -X$, and we say it is \textit{small } if $G \not\subset \bigcup_{i=1}^m (g_i + X)$ for all $m \in \N$ and all $g_1, \ldots, g_m \in G$.

\begin{definition}[\cite{D17}]
 If $G$ and $Z$ are abelian groups, and if $X \subset G$ is a small symmetric subset, then a \textit{pre-modular $G$-category relative to $(Z,X)$} is a ribbon linear category $\calC$ over a field $\Bbbk$ together with a compatible $G$-structure on $\calC$, a free realization $\sigma : \PGr \rightarrow \calC_0$, and a non-zero m-trace $\rmt$ on $\Proj(\calC)$. These data are subject to the following conditions: 
 \begin{enumerate}
  \item \textit{Generic semisimplicity}. For every $g \in G \smallsetminus X$ the homogeneous subcategory $\calC_g$ is semisimple and dominated by $\Theta(\calC_g) \otimes \sigma(Z)$ for some finite set $\Theta(\calC_g) = \{ V_i \in \calC_g \mid i \in \rmI_g \}$ of simple projective objects with epic evaluation;
  \item \textit{Compatibility}. There exists a bilinear map $\psi : G \times Z \rightarrow \Bbbk^*$ such that $c_{\sigma(k),V} \circ c_{V,\sigma(k)} = \psi(g,k) \cdot \id_{V \otimes \sigma(k)}$ for every $g \in G$, for every $V \in \calC_g$, and for every $k \in Z$.
 \end{enumerate}
\end{definition}

The group $G$ is called the \textit{structure group}, the group $Z$ is called the \textit{periodicity group}, and the set $X$ is called the \textit{critical set} of $\calC$.
Condition (1) implies that, for any $g \in G \smallsetminus X$, every object of $\cat_g$ is a direct sum of simple objects in the set $\{V_i \otimes \sigma(k) \in \calC\mid i \in \rmI_g, k \in Z \}$.
We point out that the bilinear map $\psi$ is uniquely determined by the braiding $c$ and by the free realization $\sigma$, and that, although the set of representatives $\Theta(\calC_g)$ is not unique in general, its choice does not affect
the following construction.   
In particular, both $\psi$ and $\Theta(\calC_g)$ should not be considered relevant parts of the structure of $\calC$.
Relative pre-modular $G$-categories are a slight generalization of the notion of \textit{relative $G$-modular category} introduced for the first time in \cite{CGP14}, see Section 1.5 of \cite{D17} for a full discussion of the relation between the two definitions. The change in terminology is motivated by the semisimple theory, where quantum invariants are defined for any non-degenerate pre-modular category, and modularity is an additional condition ensuring the invariant extends to a TQFT.
If $\calC$ is a pre-modular $G$-category relative to $(Z,X)$ then the associated \textit{Kirby color of index $g \in G \smallsetminus X$} is the formal linear combination of objects
\[
 \Omega_g := \sum_{i \in \rmI_g} \rmd(V_i) \cdot V_i.
\]
The name comes from Lemmas 5.9 and 5.10 of \cite{CGP14}. In particular, there exist constants $\Delta_{-\Omega},\Delta_{+\Omega} \in \Bbbk$, called \textit{stabilization coefficients}, which realize the skein equivalences of Figure \ref{F:stabilization_coefficients_Omega}, and which are independent of both $V \in \calC_g$ and $g \in G \smallsetminus X$. We say the relative pre-modular category $\calC$ is \textit{non-\-de\-gen\-er\-ate} if $\Delta_{-\Omega} \Delta_{+\Omega} \neq 0$.

\begin{figure}[tb]\label{F:stabilization_coefficients_Omega}
 \centering
 \includegraphics{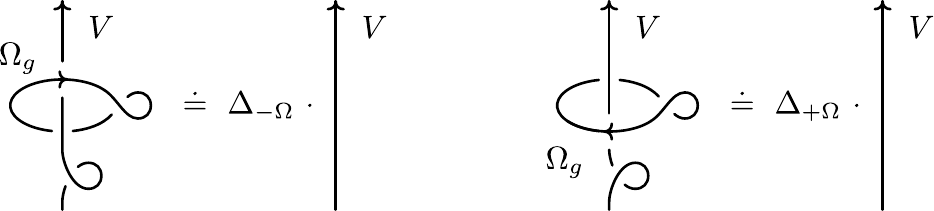}
 \caption{Skein equivalences defining $\Delta_{-\Omega}$ and $\Delta_{+\Omega}$.}
\end{figure}

In \cite{CGP14,D17} it is shown that every non-degenerate relative pre-modular category $\calC$ gives rise to a topological invariant $\rmN_{\calC}$ of admissible triples $(M,T,\omega)$, where $M$ is a closed 3-manifold, $T \subset M$ is a $\calC$-colored ribbon graph, and $\omega \in H^1(M \smallsetminus T;G)$ is a compatible cohomology class, meaning that every edge $e \subset T$ is colored with an object of $\calC_{\langle \omega,m_e \rangle}$ for the homology class $m_e$ of a positive meridian of $e$. The CGP invariant $\rmN_{\calC}$ is defined only for \textit{admissible} triples $(M,T,\omega)$, which are triples such that every component of $M$ contains either a \textit{projective} edge of $T$, that is an edge of $T$ whose color is a projective object of $\calC$, or a \textit{generic} curve for $\omega$, that is an embedded closed oriented curve whose homology class is sent to $G \smallsetminus X$ by $\omega$. Its definition uses \textit{computable} surgery presentations in $S^3$, which are surgery presentations $L = L_1 \cup \ldots \cup L_{\ell}$ of $M$ satisfying $\langle \omega,m_j \rangle \in G \smallsetminus X$ for all integers $1 \leqslant j \leqslant \ell$, where $m_j$ denotes the homology class of a meridian of the component $L_j$. We interpret a surgery presentation of $M$ which is computable with respect to some decoration $(T,\omega)$ as a $\calC$-colored ribbon graph by arbitrarily choosing orientations, and by labeling every component with the corresponding Kirby color, with index prescribed by the evaluation of $\omega$ against the homology class of a positive meridian. This is a technical complication, because arbitrary surgery presentations are not computable in general. Computable surgery presentations do exist for admissible decorated closed 3-manifolds, but only up to replacing admissible decorations via certain operations called \textit{projective} and \textit{generic stabilizations}, see Section 3.1 of \cite{D17}. The idea is to build $\rmN_{\calC}$ out of a \textit{renormalized invariant} $F'_{\calC}$ of admissible closed $\calC$-colored ribbon graphs which combines the Reshetikhin-Turaev functor $F_{\calC}$ on $\calR_{\calC}$ with the m-trace $\rmt$ on $\Proj(\calC)$. We say a closed $\calC$-colored ribbon graph is \textit{admissible} if one of his edges is projective. 
Example 3 of Section 1.5 in \cite{GP18} (see also \cite{GPT09, GKP11, GPV13}) implies the formula
\[
 F'_{\calC}(T) := \rmt_V(F_{\calC}(T_V))
\] 
defines a topological invariant of the admissible closed $\calC$-colored ribbon graph $T$, where $V \in \calC$ is projective, and where $T_V$ is a \textit{cutting presentation} of $T$, i.e. an endomorphism of $(+,V)$ in $\calR_{\calC}$ whose trace is $T$. Then, for a fixed choice of a square root $\calD_{\Omega} \in \Bbbk$ of $\Delta_{- \Omega} \Delta_{+ \Omega}$, the formula
\[
 \rmN_{\calC}(M,T,\omega) := \calD_{\Omega}^{-1-\ell} \delta_{\Omega}^{-\sigma(L)} F'_{\calC} (L \cup \tilde{T})
\]
defines a topological invariant of the admissible triple $(M,T,\omega)$ thanks to Proposition 3.1 of \cite{D17}, where $L \subset S^3$ is an $\ell$-component surgery presentation of $M$ of signature $\sigma(L)$ which is computable with respect to an admissible decoration $(\tilde{T},\tilde{\omega})$ obtained from $(T,\omega)$ by performing projective or generic stabilization, and where $\delta_{\Omega} = \calD_{\Omega} / \Delta_{-\Omega}$.

\subsection{2+1-TQFTs from relative modular categories}\label{Subs:CGP_TQFTs}

A stronger non-de\-gen\-er\-a\-cy condition is required in order to extend CGP invariants to graded TQFTs.

\begin{definition}[\cite{D17}]\label{D:relative_modular}
 A pre-modular $G$-cat\-e\-go\-ry $\calC$ relative to $(Z,X)$ is \textit{relative modular} if there exists a \textit{relative modularity parameter} $\zeta_{\Omega} \in \Bbbk^*$ realizing the skein equivalence of Figure \ref{F:relative_modularity} for all $g,h \in G \smallsetminus X$ and for all $i,j \in \rmI_g$.
\end{definition}

\begin{figure}[hbtp]\label{F:relative_modularity}
 \centering
 \includegraphics{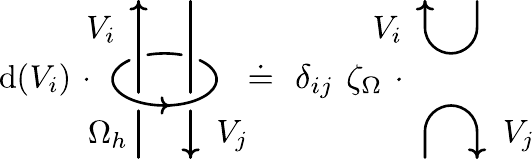}
 \caption{Relative modularity condition.}
\end{figure}

This condition automatically implies non-degeneracy, because the relative modularity parameter satisfies $\zeta_{\Omega} = \Delta_{-\Omega} \Delta_{+\Omega}$, see \cite{D17}. When $\calC$ is relative modular then, as explained in Section 6.2 of \cite{D17}, $\rmN_{\calC}$ extends to a $Z$-graded 2+1-TQFT $\bbV_{\calC}^\PGr : \adCob_{\calC}^G \rightarrow \Vect_{\Bbbk}^Z$ via a $Z$-graded refinement of the universal construction of \cite{BHMV95}, where $\adCob_{\calC}^G$ is the category of \textit{admissible} cobordisms of dimension 2+1, and where $\Vect_{\Bbbk}^Z$ is the category of $Z$-graded vector spaces. More precisely, an object of $\adCob_{\calC}^G$ is a 5-tuple $\bbSigma = (\varSigma,P,\vartheta,B,\calL)$, where $\varSigma$ is a closed surface, where $P \subset \varSigma$ is a $\calC$-colored ribbon set, where $\vartheta \in H^1(\varSigma \smallsetminus P;G)$ is a compatible cohomology class, where $B \subset \varSigma \smallsetminus P$ is a finite set composed of exactly one base point in every connected component of $\varSigma$, and where $\calL \subset H_1(\varSigma;\R)$ is a Lagrangian subspace. A morphism of $\adCob_{\calC}^G$ from $(\varSigma,P,\vartheta,B,\calL)$ to $(\varSigma',P',\vartheta',B',\calL')$ is an equivalence class of admissible 4-tuples $\bbM = (M,T,\omega,n)$, where $M$ is a 3-dimensional cobordism from $\varSigma$ to $\varSigma'$, where $T \subset M$ is a $\calC$-colored ribbon graph from $P$ to $P'$, where $\omega \in H^1(M \smallsetminus T,B \cup B';G)$ is a compatible relative cohomology class restricting to $\vartheta$ and $\vartheta'$ on the incoming and outgoing boundary of $M$ respectively, and where $n \in \Z$ is a signature defect. A 4-tuple $(M,T,\omega,n)$ is admissible if every component of $M$ which is disjoint from the incoming boundary $\partial_- M$ contains either a projective edge of $T$, or a generic curve for $\omega$, and two 4-tuples $(M,T,\omega,n)$ and $(M',T',\omega',n')$ are equivalent if $n = n'$, and if there exists a positive diffeomorphism $f : M \to M'$ which preserves boundary identifications and satisfies $f(T) = T'$ and $f^*(\omega') = \omega$. Then, $\rmN_{\calC}$ can be extended to an invariant of closed morphisms of $\adCob_{\calC}^G$ by setting
\[
 \rmN_{\calC}(M,T,\omega,n) := \delta_{\Omega}^n \rmN_{\calC}(M,T,\omega) .
\]
Remark that the category $\adCob_{\calC}^G$ thus obtained is not rigid, as objects $(\varSigma,P,\vartheta,\calL)$ such that $P$ does not contain any projective point and such that $\vartheta$ does not admit any generic curve are not dualizable.

State spaces associated with objects of $\adCob_{\calC}^G$ by the $Z$-graded TQFT $\bbV_{\calC}^\PGr$ can be described in skein theoretical terms in all degrees. We have two relevant notions of skein equivalence between morphisms of $\adCob_{\calC}^G$ from $\bbSigma$ to $\bbSigma'$, one which is local, the other which is not. Indeed, we say a formal linear combination of morphisms of $\adCob_{\calC}^G$ from $\bbSigma$ to $\bbSigma'$ is a \textit{local skein relation} if it can be written in the form
\[
 \sum_{i=1}^m \alpha_i \cdot \bbM_P \circ \left( (D^3,T_i,\omega_i,0) \disjun \id_{\bbSigma} \right)
\]
for some coefficients $\alpha_1,\ldots,\alpha_m \in \Bbbk$, for some $\calC$-colored ribbon set $P \subset S^2$ with at least one point labeled by a projective object, for some morphism $\bbM_P$ of $\adCob_{\calC}^G$ from $(S^2,P,\vartheta,\{ 0 \}) \disjun \bbSigma$ to $\bbSigma'$, and for some $\calC$-colored ribbon graphs $T_1,\ldots,T_m \subset D^3$ from $\varnothing$ to $P$ satisfying
\[
 \sum_{i=1}^m \alpha_i \cdot f_{D^3}(T_i) \doteq 0
\]
with respect to some embedding $f_{D^3} : D^3 \hookrightarrow \R^2 \times I$ mapping $P$ into $\R^2 \times \{ 1 \}$. Here the cohomology classes $\theta$ and $\omega_1,\ldots,\omega_m$ are uniquely determined by $P$ and by $T_1,\ldots,T_m$ respectively. On the other hand, we say a formal linear combination of morphisms of $\adCob_{\calC}^G$ from $\bbSigma$ to $\bbSigma'$ is a \textit{non-local skein relation} if it can be written in the form
\[
 (M,T \cup K,j^*(\omega),n) - \psi(\langle \omega,\ell_K \rangle,k) \cdot (M,T,\omega,n)
\]
for some $k \in Z$ and for some framed knot $K \subset M \smallsetminus T$ of color $\sigma(k)$, where $j^*$ is induced by the inclusion of $M \smallsetminus (T \cup K)$ into $M \smallsetminus T$, and where $\ell_K$ denotes the homology class of $K$ in $H_1(M \smallsetminus T;\Z)$. Then, if $\bbSigma = (\varSigma,P,\vartheta,B,\calL)$ is an object of $\adCob_{\calC}^G$, and if $M$ is a 3-dimensional cobordism from $\varnothing$ to $\varSigma$, the \textit{admissible skein module} $\adrmS(M;\bbSigma)$ is the quotient, induced by both local and non-local skein relations, of the free vector space $\calV(M;\bbSigma)$ generated by all pairs $(T,\omega)$ such that $(M,T,\omega,0)$ is a morphism of $\adCob_{\calC}^G$ from $\varnothing$ to $\bbSigma$. The class of a generator $(T,\omega)$ of $\calV(M;\bbSigma)$ in $\adrmS(M;\bbSigma)$ is denoted $[T,\omega]$. The proof of Proposition 4.5 in \cite{BCGP16} can be adapted to show that if $\Sigma$ and $M$ are both connected, then $\adrmS(M;\bbSigma)$ is a finite-dimensional vector space. On the other hand, if $M'$ is a 3-dimensional cobordism from $\varSigma$ to $\varnothing$, 
%
%
we denote with $\calV'(M';\bbSigma)$ the free vector space generated by all pairs $(T',\omega')$ such that $(M',T',\omega',0)$ is a morphism of $\adCob_{\calC}^G$ from $\bbSigma$ to $\varnothing$. Then, for every $k \in Z$, the degree $k$ state space $\bbV_{\calC}^k(\bbSigma)$ of a connected object $\bbSigma$ of $\adCob_{\calC}^G$ satisfies 
\[
 \bbV_{\calC}^k(\bbSigma) \cong \adrmS(M;\bbSigma \disjun \bbS^2_{-k}) / \calV'(M';\bbSigma \disjun \bbS^2_{-k})^{\perp}
\]
with respect to the pairing $\langle \cdot , \cdot \rangle_{\bbSigma \disjun \bbS^2_{-k}} : \calV'(M';\bbSigma \disjun \bbS^2_{-k}) \otimes \adrmS(M;\bbSigma \disjun \bbS^2_{-k}) \rightarrow \Bbbk$ defined by
\[
 \langle (T',\omega') , [T,\omega] \rangle_{\bbSigma \disjun \bbS^2_{-k}} := \rmN_{\calC} \left( (M',T',\omega',0) \circ (M,T,\omega,0) \right),
\]
where $M$ is a connected 3-dimensional cobordism from $\varnothing$ to $\varSigma \sqcup S^2$, where $M'$ is a connected 3-dimensional cobordism from $\varSigma \sqcup S^2$ to $\varnothing$, and where the object 
\[
 \bbS^2_{-k} = (S^2,P_{((+,V_0),(+,\sigma(-k)),(-,V_0))},\vartheta_{((+,V_0),(+,\sigma(-k)),(-,V_0))},B,\{ 0 \})
\]
of $\adCob_{\calC}^G$ is determined by the $\calC$-colored ribbon set $P_{((+,V_0),(+,\sigma(-k)),(-,V_0))} \subset S^2$ composed of three points in standard positions with orientations and colors specified by their subscript for some $g_0 \in G \smallsetminus X$ and for some $V_0 \in \Theta(\calC_{g_0})$. Remark that an explicit characterization of these quotients can sometimes be achieved. For instance, if $P \subset S^2$ is a $\calC$-colored ribbon set composed of a single positive point of color $F_\calC(P) \in \Proj(\calC_0)$, then, thanks to Remark 7.2 of \cite{D17}, the $\PGr$-graded state space $\bbV_{\calC}^\PGr(\bbS^2_P)$ of the object $\bbS^2_P = (S^2,P,0,\{ 0 \})$ of $\adCob_{\calC}^G$ satisfies
\[
 \bbV_{\calC}^\PGr(\bbS^2_P) \cong \bbHom_\calC(\one,F_\calC(P)),
\]
where for all $V,V' \in \calC$ we denote with $\bbHom_\calC(V,V')$ the $\PGr$-graded vector space whose space of degree $k$ vectors is given by $\Hom_\calC(V,V' \otimes \sigma(-k))$ for every $k \in \PGr$. See also Proposition 7.16 of \cite{D17} for a description, in terms of homogeneous colorings of trivalent graphs, of the $\PGr$-graded state space of \textit{generic surfaces} in $\adCob_{\calC}^G$.

\subsection{3-Manifold invariants from finite-dimensional non-degenerate unimodular ribbon Hopf algebras}\label{Subs:Hennings_invariants}

Next, let us move on to the renormalized Hennings theory. We start by fixing our notation for Hopf algebras, and by recalling some crucial definitions and results. If $\Bbbk$ is a field, a finite-dimensional ribbon Hopf algebra $H$ is a finite-dimensional vector space over $\Bbbk$ endowed with a multiplication $m : H \otimes H \to H$, a unit $\eta : \Bbbk \to H$, a coproduct $\Delta : H \to H \otimes H$, a counit $\varepsilon : H \to \Bbbk$, an antipode $S : H \to H$, an R-matrix $R = \sum_{i=1}^r a_i \otimes b_i \in H \otimes H$, and a ribbon element $v$ in the center of $H$, see \cite{R11} for a list of the axioms these structure maps and elements are subject to. We use the notation $m(x \otimes y) = xy$ for every $x \otimes y \in H \otimes H$ and $\eta(1) = 1$, and we denote with $u = \sum_{i=1}^r S(b_i)a_i \in H$ the Drinfeld element, and with $g = uv^{-1} \in H$ the pivotal element associated with the ribbon structure of $H$. As a consequence of finite-dimensionality, $H$ admits a right integral $\lambda \in H^*$ and a left cointegral $\Lambda \in H$ which are unique up to scalar, and we can fix a pair satisfying $\lambda(\Lambda) = 1$. The Hopf algebra $H$ is \textit{non-degenerate} if the stabilization coefficients $\Delta_{-\lambda} := \lambda(v)$ and $\Delta_{+\lambda} := \lambda(v^{-1})$ satisfy $\Delta_{-\lambda} \Delta_{+\lambda} \neq 0$, and it is \textit{unimodular} if the left cointegral $\Lambda$ is two-sided, meaning that it is also a right cointegral. The category $\calC = H$-mod of finite-dimensional left $H$-modules is a ribbon linear category, with evaluation and coevaluation morphisms given by
\begin{align*}
 &\begin{array}{rccc}
   \lev_V : & V^* \otimes V & \rightarrow & \Bbbk \\
   & f \otimes v & \mapsto & f(v) \vphantom{\displaystyle \sum_{i=1}^n}
  \end{array} &
 &\begin{array}{rccc}
   \lcoev_V : & \Bbbk & \rightarrow & V \otimes V^* \\
   & 1 & \mapsto & \displaystyle \sum_{i=1}^n v_i \otimes f_i 
  \end{array} \\
  &\begin{array}{rccc}
   \rev_V : & V \otimes V^* & \rightarrow & \Bbbk \\
   & v \otimes f & \mapsto & f(\rho_V(g)(v)) \vphantom{\displaystyle \sum_{i=1}^n}
  \end{array} &
 &\begin{array}{rccc}
   \rcoev_V : & \Bbbk & \rightarrow & V^* \otimes V \\
   & 1 & \mapsto & \displaystyle \sum_{i=1}^n f_i \otimes \rho_V(g^{-1})(v_i)
  \end{array}
\end{align*}
for every left $H$-module $V$ with basis $\{ v_1,\ldots,v_n \}$ and dual basis $\{ f_1,\ldots,f_n \}$, and with braiding morphisms given by
\[
 \begin{array}{rccc}
  c_{V,V'} : & V \otimes V' & \rightarrow & V' \otimes V \\
  & v \otimes v' & \mapsto & \displaystyle \sum_{i=1}^r \rho_{V'}(b_i)(v') \otimes \rho_V(a_i)
 \end{array}
\]
for all left $H$-modules $V$ and $V'$. Thanks to Theorem 1 of \cite{BBG18}, $\calC$ admits an m-trace $\rmt$ on the ideal of projective $H$-modules $\Proj(\calC)$, which is unique up to scalar and uniquely determined by the condition $\rmt_H(f) = \lambda(gf(1))$ for all $f \in \End_\calC(H)$, where $H \in \Proj(\calC)$ denotes the regular representation of $H$. Furthermore, $\rmt$ is non-degenerate.

In \cite{DGP18} it is shown that every finite-dimensional non-degenerate ribbon Hopf algebra $H$ gives rise to a topological invariant $\rmH'_{\calC}$ of admissible pairs $(M,T)$, where $M$ is a closed 3-manifold, and $T \subset M$ is a $\calC$-colored bichrome graph. The latter are $\calC$-colored ribbon graphs carrying a set of specified edges, and their name comes from the fact that we think about special edges as being red, while the rest of the graph is blue. Red edges can only be colored with the regular representation $H$, and they can only intersect coupons in a prescribed way: for every coupon of a bichrome graph there exists an integer $k \geqslant 0$ such that the first $k$ input edges are incoming and red, the first $k$ output edges are outgoing and red, while all the other ones are blue. Such a coupon is colored with a morphism in the $k$-th stabilized subcategory $[k]\calC$ of $\calC$, which is the category whose objects have the form $[k]V := H^{\otimes k} \otimes V \in \calC$ for some $V \in \calC$, and whose morphisms have the form $\sum_{i=1}^m L_{\underline{x_i}} \otimes f_i \in \Hom_{\calC}([k]V,[k]V')$ for some left translation $L_{\underline{x_i}} \in \End_\Bbbk(H^{\otimes k})$ by $\underline{x_i} \in H^{\otimes k}$ and for some linear map $f_i \in \Hom_{\Bbbk}(V,V')$. The ribbon category $\calR_{\lambda}$ of $\calC$-colored bichrome graphs provides a graphical calculus which is formalized by the \textit{Hennings-Reshetikhin-Turaev functor} $F_{\lambda} : \calR_{\lambda} \to \calC$ introduced in Proposition 2.5 of \cite{DGP18}. By definition, $F_{\lambda}$ coincides with the Reshetikhin-Turaev functor in the absence of red edges, it coincides with the Hennings invariant in the absence of blue edges, and it coherently combines the two behaviors for general $\calC$-colored bichrome graphs. Remark that $F_{\lambda}$ yields a notion of \textit{skein equivalence} between formal linear combinations of $\calC$-colored bichrome graphs in the same way $F_{\calC}$ does for $\calC$-colored ribbon graphs. This way, $\calR_{\calC}$ is naturally identified with the subcategory of $\calR_{\lambda}$ whose morphisms are entirely blue. The renormalized Hennings invariant $\rmH'_{\calC}$ is then defined only for \textit{admissible} pairs $(M,T)$, which are pairs such that every component of $M$ contains a \textit{projective} blue edge of $T$. Its definition uses surgery presentations in $S^3$, which we interpret as $\calC$-colored bichrome graphs by arbitrarily choosing orientations, by labeling every component with the regular representation $H$, and by taking them to be red. The idea is to build $\rmH'_{\calC}$ out of a \textit{renormalized invariant} $F'_{\lambda}$ of admissible closed $\calC$-colored bichrome graphs which combines the Hennings-Reshetikhin-Turaev functor $F_{\lambda}$ on $\calR_{\lambda}$ with the m-trace $\rmt$ on $\Proj(\calC)$. We say a closed $\calC$-colored bichrome graph is \textit{admissible} if one of his blue edges is projective. The formula
\[
 F'_{\lambda}(T) := \rmt_V(F_{\lambda}(T_V))
\] 
defines a topological invariant of the admissible closed $\calC$-colored bichrome graph $T$ thanks to Theorem 2.7 of \cite{DGP18}, where $V$ is a projective object of $\calC$, and where $T_V$ is a \textit{cutting presentation} of $T$, meaning an endomorphism of $(+,V)$ in $\calR_{\lambda}$ whose trace is $T$. Then, for a fixed choice of a square root $\calD_{\lambda} \in \Bbbk$ of $\Delta_{- \lambda} \Delta_{+ \lambda}$, the formula
\[
 \rmH'_{\calC}(M,T) := \calD_{\lambda}^{-1-\ell} \delta_{\lambda}^{-\sigma(L)} F'_{\lambda} (L \cup T)
\]
defines a topological invariant of the admissible pair $(M,T)$ thanks to Theorem 2.9 of \cite{DGP18}, where $L \subset S^3$ is an $\ell$-component surgery presentation of $M$ of signature $\sigma(L)$, and where $\delta_{\lambda} = \calD_{\lambda} / \Delta_{-\lambda}$.

\subsection{2+1-TQFTs from finite-dimensional factorizable ribbon Hopf algebras}\label{Subs:Hennings_TQFTs}

A finite-dimensional ribbon Hopf algebra $H$ with R-matrix $R = \sum_{i=1}^r a_i \otimes b_i$ is \textit{factorizable} if the \textit{Drinfeld map} $\psi_H : H^* \to H$, which is defined by
\[
 \psi_H(f) := \sum_{i,j=1}^r f(b_ja_i) \cdot a_jb_i
\] 
for every $f \in H^*$, is an isomorphism. This condition implies both non-degeneracy and unimodularity, see \cite{H96} and \cite{R11}. When $H$ is factorizable then, as explained in Section 3 of \cite{DGP18}, $\rmH'_{\calC}$ extends to a 2+1-TQFT $\rmV_{\calC} : \adCob_{\calC} \rightarrow \Vect_{\Bbbk}$ via the universal construction of \cite{BHMV95}, where $\adCob_{\calC}$ is the category of admissible cobordisms of dimension 2+1. More precisely, an object of $\adCob_{\calC}$ is a triple $\bbSigma = (\varSigma,P,\calL)$, where $\varSigma$ is a closed surface, where $P \subset \varSigma$ is a blue $\calC$-colored ribbon set, and where $\calL \subset H_1(\varSigma;\R)$ is a Lagrangian subspace. A morphism of $\adCob_{\calC}$ from $(\varSigma,P,\calL)$ to $(\varSigma',P',\calL')$ is an equivalence class of admissible triples $\bbM = (M,T,n)$, where $M$ is a 3-dimensional cobordism from $\varSigma$ to $\varSigma'$, where $T \subset M$ is a $\calC$-colored bichrome graph from $P$ to $P'$, and where $n \in \Z$ is a signature defect. A triple $(M,T,n)$ is \textit{admissible} if every component of $M$ which is disjoint from the incoming boundary $\partial_- M$ contains a projective blue edge of $T$, and two triples $(M,T,n)$ and $(M',T',n')$ are equivalent if $n = n'$, and if there exists a positive diffeomorphism $f : M \to M'$ which preserves boundary identifications and satisfies $f(T) = T'$. Then, $\rmH'_{\calC}$ can be extended to an invariant of closed morphisms of $\adCob_{\calC}$ by setting
\[
 \rmH'_{\calC}(M,T,n) := \delta_{\lambda}^n \rmH'_{\calC}(M,T) .
\]
Remark that the category $\adCob_{\calC}$ thus obtained is not rigid, as objects $(\varSigma,P,\calL)$ such that $P$ does not contain any projective blue point are not dualizable.

State spaces associated with objects of $\adCob_{\calC}$ by the TQFT $\rmV_{\calC}$ can be presented as quotients of admissible skein modules, just like we did in the CGP case, and of course this time only local skein relations are needed. However, they can also be efficiently described in terms of the \textit{dual coadjoint $H$-module $X$}, which is the vector space $H$ equipped with the action $\rho_X(h)(x) := h_{(2)}xS^{-1}(h_{(1)})$ for all $h \in H$ and $x \in X$. Indeed, if $V$ is a left $H$-module with action $\rho_V : H \rightarrow \End_{\Bbbk}(V)$, we can consider its subspace of $H$-invariant vectors, which is defined as 
\[
 V^H := \{ v \in V \mid \rho_V(h)(v) = \varepsilon(h) \cdot v \ \Forall h \in H \}.
\]
Remark that we have an obvious isomorphism between $\Hom_{\calC}(\one,V)$ and $V^H$ sending $f$ to $f(1)$. Then, if $\varSigma_g$ is a closed surface of genus $g \in \N$, and if $P_V \subset \varSigma_g$ is a single positive framed blue point of color $V \in \calC$, it follows directly from Corollary 3.21 of \cite{DGP18} that the state space of the object $\bbSigma_{g,V} = (\varSigma_g,P_V,\calL)$ of $\adCob_{\calC}$ determined by any arbitrary Lagrangian $\calL \subset H_1(\varSigma_g;\R)$ satisfies
\[
 \rmV_{\calC}(\bbSigma_{g,V}) \cong \bp{(V^* \otimes X^{\otimes g})^H}^*.
\]

\subsection{Main results}

As we mentioned earlier, this paper contains two main results related
to the non-semisimple constructions we just recalled. The first one
concerns the existence of a family of graded TQFTs, as well as graded
ETQFTs, for the CGP theory. The setting is provided by unrolled
quantum groups at odd roots of unity. More precisely, in Subsection
\ref{Subs:unrolled_quantum_groups} we recall the definition, for every
simple complex Lie algebra $\frakg$ of rank $n$ and dimension
$2N + n$, of a particular quantum deformation, denoted
$U^H_q(\frakg)$, of the enveloping algebra $U(\frakg)$ for
$q = \smash{e^{\frac{2 \pi i}{r}}}$, where $r \geq 3$ is an odd
integer which is required not to be a multiple of 3 when
$\frakg = \frakg_2$. These unrolled quantum groups are quite different
from the ones which usually underlie quantum constructions in low
dimensional topology. For instance, they are infinite dimensional:
indeed, they are generated by
$\{ E_i,F_i,H_i,K_i,K_i^{-1} \mid 1 \leq i \leq n \}$, but while
generators $E_i$ and $F_i$, as well as their induced root vectors, are
set to be nilpotent, generators $K_i$ and $K_i^{-1}$ are not required
to be 
quasi-unipotent. 
This produces a representation theory in which
weights are allowed to take arbitrary complex values, instead of
integral ones. As a consequence, we need to be careful when it comes
to defining braidings of representations. Indeed, we need the presence
of generators $H_i$, which should be thought of a logarithms of
generators $K_i$. This exponential relation is not set at the level of
the quantum group, but we restrict to representations where it is
satisfied. More precisely, we focus on the full subcategory $\calC^H$
of finite-dimensional representations of $U^H_q \frakg$ where the
action of generators $H_i$ is diagonalizable, and where the action of
generators $K_i$ is obtained by exponentiating. This category is
non-semisimple, and it was studied in detail in \cite{GP13}: a full
subcategory $\calD^\vartheta$ of $\calC^H$ was proven to be ribbon,
and the equality $\calD^\vartheta = \calC^H$ was conjectured. In
\cite{CGP14}, it was shown that $\calD^\vartheta$ is non-degenerate
relative pre-modular, and thus yields a quantum invariant
$\rmN_{\calC^H}$ of admissible decorated 3-manifolds. In \cite{GP18},
the conjecture was proven: $\calC^H$ is a relative pre-modular
category. Its structure group $G$ is given by $\frakh^* / \Lambda_R$,
where $\frakh$ is a Cartan subalgebra of $\frakg$ with root lattice
$\Lambda_R$. Its periodicity group $Z$ is given by
$\Lambda_R \cap (r \cdot \Lambda_W)$, where $\Lambda_W$ denotes the
weight lattice. Its critical set $X$ is given by
$\{ [\xi] \in \frakh^* / \Lambda_R \mid \Exists \alpha \in \Phi_+ : 2
\langle \alpha,\xi \rangle \in \Z \}$, where $\Phi_+$ is a set of
positive roots of $\frakg$. The following is our first main result,
which implies, as an immediate consequence, the existence of a
$Z$-graded ETQFT in dimension 1+1+1 extending the quantum invariant
$\rmN_{\calC^H}$.

\begin{theorem}\label{T:relative_modularity}
 The category $\calC^H$ is relative modular.
\end{theorem}

The second main result of this paper builds a bridge between this
family of quantum invariants and the renormalized Hennings ones coming
from the corresponding small quantum groups, under the additional
assumption that $\gcd(r,\det(A)) = 1$, where $A$ denotes the Cartan
matrix of $\frakg$. Indeed, in Subsection
\ref{Subs:small_quantum_groups} we recall the definition of a more
classical quantum version of $U(\frakg)$, denoted $\bar{U}_q(\frakg)$,
again for $q = \smash{e^{\frac{2 \pi i}{r}}}$. These \textit{small}
quantum groups are far better known: they are finite-dimensional, as
generators $K_i$ and $K_i^{-1}$ are set to be
quasi-unipotent, they are
ribbon and factorizable, and thus they yield TQFTs in dimension
2+1. The category $\bar{\calC}$ of finite-dimensional representations
of $\bar{U}_q \frakg$ is still non-semisimple, but all weights take
integral values. Indeed, we have a very natural forgetful functor
$\Phi_{\calC}$ from the full subcategory $\smash{\calC^H_{[0]}}$ of
$\calC^H$ whose objects have all weights in $\Lambda_R$ to
$\bar{\calC}$: the image $\bar{V}$ of an object $V$ of
$\smash{\calC^H_{[0]}}$ is simply defined by forgetting the action of
generators $H_i$. The technical condition $\gcd(r,\det(A))=1$
ensures
$\Phi_{\calC}$ is essentially
surjective. Furthermore, $\Phi_{\calC}$ immediately induces a functor
$\Phi_{\calR}$ from $\calR_{\smash{\calC^H_{[0]}}}$ to
$\calR_{\bar{\calC}}$: the image $\bar{T}$ of a morphism $T$ of
$\calR_{\smash{\calC^H_{[0]}}}$ is simply defined by applying the
forgetful functor $\Phi_{\calC}$ to all its colors.

\begin{theorem}\label{T:main_result}
 If $M$ is a closed 3-manifold and $T \subset M$ is an admissible $\calC^H_{[0]}$-colored ribbon graph, then
 \[
  \rmN_{\calC^H}(M,T,0) = \rmH'_{\bar{\calC}}(M,\bar{T}).
 \]
\end{theorem}

We point out that each invariant depends on the choice of a square
root of the product of the stabilization coefficients for the
corresponding version of the quantum group. Theorem
\ref{T:main_result} requires a coherent choice of these square roots
in the two theories, otherwise a sign will appear in the
relation. Remark also that a great advantage of the renormalized
Hennings theory is the absence of many technical complications which
characterize the CGP one. For instance, arbitrary surgery
presentations can be used to define and compute quantum invariants
associated with small quantum groups, as we have no computability
condition. Therefore, Theorem \ref{T:main_result} gives a convenient
alternative formulation for the CGP invariants associated with
unrolled quantum groups, at least in the case of 3-manifolds decorated
with the trivial cohomology class.

Let us end this introduction with a final comment 
about our choice for the setting of Theorem
\ref{T:main_result}. When we started writing these results, the
renormalized Hennings construction had only been developed in the case
of finite-dimensional factorizable ribbon Hopf algebras. As a
consequence, a comparison with the CGP construction had to take place
in this context.
The classification of finite-dimensional factorizable quantum groups is
not so easy, see \cite{LO17}, and we choose for simplicity to limit 
ourselves to the smallest set of Cartan generators $\{ K_i \mid 1 \leq i \leq n \}$
at odd roots of unity.
However, a recent generalization of the renormalized Hennings
construction allows us to consider arbitrary modular categories, in
the non-semisimple sense, as building blocks for non-semisimple
2+1-TQFTs \cite{DGGPR}. This larger setting encompasses examples of
ribbon categories coming from the representation theory of restricted
quantum groups at even roots of unity. Indeed, while
some of these Hopf algebras might not be braided \cite{LN15,LO17}, a
suitable modification of their coalgebra structure produces quasi-Hopf
deformations which admit a ribbon structure
\cite{CGR19,GLO18,N18}. For what concerns unrolled quantum groups,
even roots of unity have been less studied due to several technical
difficulties which produce fascinating but complicated phenomena in
the theory. It would be very interesting to generalize Theorem
\ref{T:main_result} to this setting, and possibly to even larger ones.

\section{Quantum groups at odd roots of unity}

In this section we recall definitions of small and unrolled quantum groups associated with arbitrary simple complex Lie algebras $\frakg$, and we prove our first result: categories of finite-dimensional weight representations of unrolled quantum groups at odd roots of unity are relative modular, and can therefore be used to construct ETQFTs in dimension 1+1+1.

\subsection{Small quantum groups}\label{Subs:small_quantum_groups}

Let $\frakg$ be a simple complex Lie algebra of rank $n$ and dimension $2N + n$, let $\frakh$ be a Cartan subalgebra of $\frakg$, let $\Phi_+$ be a set of positive roots of $\frakg$, and let $\calU_q \frakg$ be the associated quantum group, over a formal parameter $q$, introduced in Appendix \ref{A:quantum_groups}. Let us fix an odd integer $r \geq 3$ with the further condition that $r \not\equiv 0$ modulo $3$ if $\frakg = \frakg_2$, and let us specialize $q$ to $e^{\frac{2 \pi i}{r}}$ in the De Concini-Kac sense 
\cite{DK90}. Let $\bar{U}_q \frakg$ denote the \textit{small quantum group of $\frakg$}, which is the $\C$-algebra obtained from $\calU_q \frakg$ by adding relations
\[
 K_{\mu} = 1, \quad E_{\alpha}^r = F_{\alpha}^r = 0
\]
for every $\mu \in \Lambda_R \cap r \cdot \Lambda_W$ and every $\alpha \in \Phi_+$. Then $\bar{U}_q \frakg$ inherits from $\calU_q \frakg$ the structure of a Hopf algebra, and we denote with $\bar{U}_q \frakh$, with $\bar{U}_q \frakn_+$, and with $\bar{U}_q \frakn_-$ the subalgebras of $\bar{U}_q \frakg$ generated by $\{ K_i \mid 1 \leq i \leq n \}$, by $\{ E_i \mid 1 \leq i \leq n \}$, and by $\{ F_i \mid 1 \leq i \leq n \}$, respectively. As proved in \cite{L90}, see also Theorem 30 of \cite{GP13}, a Poincaré-Birkhoff-Witt basis is given by
\[
 \left\{ \left( \prod_{k = 1}^N F_{\beta_k}^{c_k} \right) K_{\mu} \left( \prod_{k = 1}^N E_{\beta_k}^{b_k} \right) \Biggm| 
 \begin{array}{l}
  \mu \in \Lambda_R / (\Lambda_R \cap r \cdot \Lambda_W), \\
  0 \leq b_1,\ldots,b_N < r, \\
  0 \leq c_1,\ldots,c_N < r
 \end{array} \right\},
\]
so $\bar{U}_q \frakg$ is finite-dimensional. A pivotal element is given by $K_{2 \cdot \rho} \in \bar{U}_q \frakg$ where
\[
 \rho := \frac{1}{2} \cdot \sum_{k=1}^N \beta_k.
\]
Furthermore, if we consider $\bar{R}_0 \in \bar{U}_q \frakh \otimes \bar{U}_q \frakh$ given by
\[
 \bar{R}_0 := \frac{1}{\left| \Lambda_R / (\Lambda_R \cap r \cdot \Lambda_W) \right|} \cdot \sum_{\mu,\mu' \in \Lambda_R / (\Lambda_R \cap r \cdot \Lambda_W)} q^{- \langle \mu,\mu' \rangle} \cdot 
 K_{\mu} \otimes K_{\mu'}
\]
and $\bar{\Theta} \in \bar{U}_q \frakn_+ \otimes \bar{U}_q \frakn_-$ given by
\[
 \bar{\Theta} := \sum_{b_1,\ldots,b_N=0}^{r-1} 
 \left( \prod_{k=1}^N \frac{\{ 1 \}_{\beta_k}^{b_k}}{[b_k]_{\beta_k}!} q^{\frac{b_k(b_k-1)}{2}} \right) \cdot 
 \left( \prod_{k=1}^N E_{\beta_k}^{b_k} \right) \otimes 
 \left( \prod_{k=1}^N F_{\beta_k}^{b_k} \right)
\]
then $\bar{R} := \bar{R}_0 \bar{\Theta} \in \bar{U}_q \frakg \otimes \bar{U}_q \frakg$ is an R-matrix for $\bar{U}_q \frakg$, as proved in \cite{L90}, see also \cite{LO17}. Next, thanks to Proposition A.5.1 of \cite{L95}, a right integral $\lambda$ of $\bar{U}_q \frakg$ is given by
\[
 \lambda \left( \left( \prod_{k = 1}^N F_{\beta_k}^{c_k} \right) K_{\mu} \left( \prod_{k = 1}^N E_{\beta_k}^{b_k} \right) \right) = q^{- 4 \langle \rho,\rho \rangle} \delta_{\mu,2 \cdot \rho} 
 \prod_{k=1}^N \delta_{b_k,r-1} \prod_{k=1}^N \delta_{c_k,r-1}.
\]
This formula can be deduced from the one in \cite{L95} by remarking that Lyubashenko uses Luszitg's coproduct $\tilde{\Delta} := (\omega \otimes \omega) \circ \Delta^\op \circ \omega$, where $\omega$ denotes the involutive algebra automorphism of $\bar{U}_q \frakg$ defined by $\omega(E_i) = F_i$ and by $\omega(K_i) = K_i^{-1}$ for all integers $1 \leq i \leq n$, and by remarking that $\lambda$ is a right integral for $\Delta$ if and only if $\lambda \circ \omega$ is a left integral for $\tilde{\Delta}$. Finally, thanks to Proposition A.5.2 of \cite{L95}, a two-sided cointegral $\Lambda$ of $\bar{U}_q \frakg$ satisfying $\lambda(\Lambda) = 1$ is given by
\[
 \Lambda := \sum_{\mu \in \Lambda_R / (\Lambda_R \cap r \cdot \Lambda_W)} q^{2 \langle \mu,\rho \rangle} \cdot \left( \prod_{k = 1}^N F_{\beta_k}^{r-1} \right) K_{\mu} \left( \prod_{k = 1}^N E_{\beta_k}^{r-1} \right).
\]

\begin{proposition}
 The Hopf algebra $\bar{U}_q \frakg$ is factorizable and ribbon.
\end{proposition}

This result is proved in \cite{L95}, see also \cite{LO17}. We denote with $\bar{\calC}$ the ribbon category of finite-dimensional $\bar{U}_q \frakg$-modules, and with $\bar{\rmt}$ the m-trace on $\Proj(\bar{\calC})$ given by Theorem 1 of \cite{BBG18}, which satisfies $\bar{\rmt}_{\bar{U}}(\Lambda \circ \varepsilon) = 1$ for the regular representation $\bar{U}$ of $\bar{U}_q \frakg$.

\begin{corollary}
 The renormalized Hennings invariant $\rmH'_{\bar{\calC}}$ extends to a TQFT
 \[
  \rmV_{\bar{\calC}} : \adCob_{\bar{\calC}} \rightarrow \Vect_{\C}.
 \]
\end{corollary}

\subsection{Unrolled quantum groups}\label{Subs:unrolled_quantum_groups}

Let $U^H_q \frakg$ denote the \textit{unrolled quantum group of $\frakg$}, which is the $\C$-algebra obtained from $\calU_q \frakg$ by adding generators
\[
 \{ H_i \mid 1 \leqslant i \leqslant n \}
\]
and relations
\[
 [H_i,H_j] = [H_i,K_j] = 0, \quad [H_i,E_j] = a_{ij} E_j, \quad [H_i,F_j] = - a_{ij} F_j, \quad E_{\alpha}^r = F_{\alpha}^r = 0
\]
for every integer $1 \leqslant i,j \leqslant n$ and every positive root $\alpha \in \Phi_+$. Then $U^H_q \frakg$ can be made into a pivotal Hopf algebra by setting
\[
 \Delta(H_i) =  H_i \otimes 1 + 1 \otimes H_i, \quad \varepsilon(H_i) = 0, \quad S(H_i) = -H_i
\]
for every integer $1 \leqslant i \leqslant n$, and we denote with $U^H_q \frakh$, with $U^H_q \frakn_+$, and with $U^H_q \frakn_-$ the subalgebras of $U^H_q \frakg$ generated by $\{ H_i \mid 1 \leqslant i \leqslant n \}$, by $\{ E_i \mid 1 \leqslant i \leqslant n \}$, and by $\{ F_i \mid 1 \leqslant i \leqslant n \}$ respectively. For every $z \in \C$ let us introduce the notation
\[
 q^z := e^{\frac{z 2 \pi i}{r}}, \quad \{ z \} := q^z - q^{-z}.
\]
A $U^H_q \frakg$-module $V$ with action $\rho_V : U^H_q \frakg \rightarrow \End_{\Bbbk}(V)$ is a \textit{weight module} if it is a semisimple $U^H_q \frakh$-module and if for every $\mu \in \frakh^*$ and every $v \in V$ we have
\[
 \rho_V(H_i)(v) = \mu(H_i) \cdot v \quad \Forall 1 \leqslant i \leqslant n \quad \Rightarrow \quad \rho_V(K_i)(v) = q_i^{\mu(H_i)} \cdot v \quad \Forall 1 \leqslant i \leqslant n,
\]
where we are identifying $\frakh$ with the corresponding linear subspace of $U^H_q \frakh$ in the obvious way. We denote with $\calC^H$ the full subcategory of the category of finite-dimensional $U^H_q \frakg$-modules whose objects are weight modules. Then $\calC^H$ can be made into a ribbon category as follows: first of all, a pivotal element is given by $K_{2 \cdot \rho}^{1-r} \in U_q^H \frakg$, where the choice of the exponent is explained in Remark 4 of \cite{GP18}. Furthermore, if $V$ and $V'$ are objects of $\calC^H$, their braiding morphism is given by 
\[
 \begin{array}{rccc}
  c_{V,V'} : & V \otimes V' & \rightarrow & V' \otimes V \\
  & v \otimes v' & \mapsto & \tau_{V,V'}(R^H_{0,V,V'}((\rho_V \otimes \rho_{V'})(\Theta^H)(v \otimes v')))
 \end{array}
\]
for the linear maps $R^H_{0,V,V'} : V \otimes V' \rightarrow V \otimes V'$ and $\tau_{V,V'} : V \otimes V' \rightarrow V' \otimes V$ determined by
\[
 R^H_{0,V,V'}(v \otimes v') := q^{\langle \nu,\nu' \rangle} \cdot v \otimes v', \quad \tau_{V,V'}(v \otimes v') := v' \otimes v
\]
for all $v \in V$, $v' \in V'$ satisfying
\[
 \rho_V(H_i)(v) = \nu(H_i) \cdot v, \quad \rho_{V'}(H_i)(v') = \nu'(H_i) \cdot v'
\]
for every integer $1 \leqslant i \leqslant n$, and for $\Theta^H = \bar{\Theta} \in U^H_q \frakn_+ \otimes U^H_q \frakn_- \cong \bar{U}_q \frakn_+ \otimes \bar{U}_q \frakn_-$. Thanks to Theorem 4 of \cite{GP18}, $\calC^H$ is a ribbon category.

If we set $G := \frakh^* / \Lambda_R$, then $\calC^H$ supports the structure of a $G$-category: indeed, for every $\gamma \in \frakh^*$ we can define the homogeneous subcategory $\calC^H_{[\gamma]}$ to be the full subcategory of $\calC^H$ with objects given by modules whose weights are all of the form $\gamma + \mu$ for some $\mu \in \Lambda_R$. Furthermore, if we set $\PGr := \Lambda_R \cap (r \cdot \Lambda_W)$, then we have a free realization $\sigma : \PGr \rightarrow \calC^H_{[0]}$ mapping every $\kappa \in \PGr$ to the object $\sigma(\kappa) \in \calC^H_{[0]}$ given by the vector space $\C$ with $U^H_q \frakg$-action specified by
\[
 \rho_{\sigma(\kappa)}(H_i)(1) := \kappa(H_i), \quad \rho_{\sigma(\kappa)}(E_i)(1) := 0, \quad \rho_{\sigma(\kappa)}(F_i)(1) := 0
\]
for every integer $1 \leqslant i \leqslant n$.
%
%
Now the bilinear map 
\[
 \begin{array}{rccc}
  \psi : & G \times \PGr & \rightarrow & \C^* \\
  & ([\gamma],\kappa) & \mapsto & q^{2\langle \gamma,\kappa \rangle}
 \end{array}
\]
satisfies $c_{\sigma(\kappa),V} \circ c_{V,\sigma(\kappa)} = \psi([\gamma],\kappa) \cdot \id_{V \otimes \sigma(\kappa)}$ for every $\gamma \in \frakh^*$, every $V \in \calC^H_{[\gamma]}$, and every $\kappa \in \PGr$. If we consider the critical set
\[
 X := \{ [\xi] \in \frakh^* / \Lambda_R \mid \Exists \alpha \in \Phi_+ : 2 \langle \alpha,\xi \rangle \in \Z \}
\]
then, as explained in Section 7 of \cite{CGP14}, the category $\calC^H_{[\gamma]}$ is semisimple for every $[\gamma] \in G \smallsetminus X$. Therefore, the last relevant piece of structure we are missing is an m-trace. In order to define it, let us introduce typical $U^H_q\frakg$-modules. First of all, we say a vector $v_+$ of a $U^H_q\frakg$-module $V$ is a \textit{highest weight vector} if $\rho_V(E_i)(v_+) = 0$ for every integer $1 \leqslant i \leqslant n$. Analogously, we say a vector $v_-$ of $V$ is a \textit{lowest weight vector} if $\rho_V(F_i)(v_-) = 0$ for every integer $1 \leqslant i \leqslant n$. Then for every weight $\mu \in \frakh^*$ there exists a simple finite-dimensional weight $U^H_q\frakg$-module $V_{\mu}$ featuring a highest weight vector of weight $\mu$. This module is unique up to isomorphism, and every simple weight $U^H_q\frakg$-module is of this form, see Proposition 33 of \cite{GP13}. Every such module also has a lowest weight vector, and it is called \textit{typical} if its lowest weight is given by $\mu - 2(r-1) \cdot \rho$. If we consider the set 
\[
 \ddot{\frakh}^* := \{ \gamma \in \frakh^* \mid 2 \langle \alpha,\gamma + \rho \rangle + m \langle \alpha,\alpha \rangle \not\in r \Z \ \Forall \alpha \in \Phi_+, \ \Forall 1 \leqslant m \leqslant r-1 \}
\]
then, thanks to Proposition 34 of \cite{GP13}, $V_{\gamma}$ is typical if and only if $\gamma \in \ddot{\frakh}^*$. Remark that if $\gamma \in \frakh^*$ satisfies $2 \langle \alpha,\gamma \rangle \not\in \Z$ for every $\alpha \in \Phi$, then $\gamma \in \ddot{\frakh}^*$. This means that if $\gamma \in \frakh^*$ satisfies $[\gamma] \not\in X$, then $V_{\gamma}$ is typical. We also point out that, although $[(r-1) \cdot \rho] \in X$, the module $V_{(r-1) \cdot \rho}$ is always typical, because
\[
 2 \langle \alpha,(r-1) \cdot \rho + \rho \rangle + m \langle \alpha,\alpha \rangle = 2r \langle \alpha,\rho \rangle + m d_{\alpha}
\]
is not in $r \Z$ for any integer $1 \leq m \leq r-1$. Now, thanks to Lemma 7.1 of \cite{CGP14} and Theorem 38 of \cite{GP13}, every typical $U^H_q\frakg$-module is projective and ambidextrous. Then, by combining Theorem 3.3.2 of \cite{GKP11} with Lemma 17 of \cite{GPV13}, there exists a non-zero m-trace on the ideal $\Proj(\calC^H)$ of projective objects of $\calC^H$ which is unique up to scalar. Therefore, we can fix the normalization $\rmd^H(V_{(r-1) \cdot \rho}) = 1$. Thanks to Equation (51) and Lemma 47 of \cite{GP13}, for every $\mu \in \ddot{\frakh}^*$ we have
\[
  \rmd^H(V_{\mu}) = \prod_{k=1}^N \frac{r \{ \langle \mu - (r-1) \cdot \rho,\beta_k \rangle \}}{\{ r\langle \mu - (r-1) \cdot \rho,\beta_k \rangle \}}.
\]
For all $\mu,\nu \in \ddot{\frakh}^*$, if $f_{\mu,\nu}^+ := F_{\calC^H}(T_{\mu,\nu}^+)$ and $f_{\mu,\nu}^- := F_{\calC^H}(T_{\mu,\nu}^-)$ for the $\calC^H$-colored ribbon graphs $T_{\mu,\nu}^+$ and $T_{\mu,\nu}^-$ represented in Figure \ref{F:long_Hopf_links}, Proposition 45 of \cite{GP13} gives
\[
 \rmt^H_{V_{\nu}}(f_{\mu,\nu}^+) = r^N q^{2 \langle \mu - (r-1) \cdot \rho,\nu - (r-1) \cdot \rho \rangle},
\]
and, since $V_{\mu}^* \cong V_{2(r-1) \cdot \rho - \mu}$, it also gives
\[
 \rmt^H_{V_{\nu}}(f_{\mu,\nu}^-) = r^N q^{-2 \langle \mu - (r-1) \cdot \rho,\nu + (1-r) \cdot \rho \rangle}
 = r^{2N} \rmt^H_{V_{\nu}}(f_{\mu,\nu}^+)^{-1}.
\]

\begin{figure}[htb]
 \centering
 \includegraphics{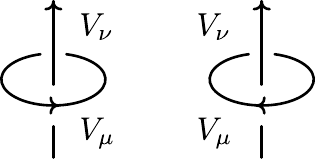}
 \caption{$\calC^H$-colored ribbon graphs $T_{\mu,\nu}^+$ and $T_{\mu,\nu}^-$.}
 \label{F:long_Hopf_links}
\end{figure}

\subsection{Relative modularity}\label{Subs:relative_modularity}

In this subsection we will prove the category $\calC^H$ is relative modular, and thus yields a $\PGr$-graded TQFT. In order to do this, we will first need a preliminary definition. We say an endomorphism $f \in \End_{\calC^H}(V)$ of an object $V$ of $\calC^H_{[0]}$ is \textit{transparent in $\calC^H_{[0]}$} if for all objects $U,W \in \calC^H_{[0]}$ we have 
\[
 \id_U \otimes f = c_{V,U} \circ (f \otimes \id_U) \circ c_{U,V}, \quad
 f \otimes \id_W = c_{W,V} \circ (\id_W \otimes f) \circ c_{V,W}.
\]

\begin{lemma}\label{L:transparency}
 If $f \in \End_{\calC^H}(V)$ is transparent in $\calC^H_{[0]}$, then there exist some integer $m$ and some morphisms $g_i \in \Hom_{\calC^H}(V,\sigma(\kappa_i))$ and $h_i \in \Hom_{\calC^H}(\sigma(\kappa_i),V)$ for every integer $1 \leqslant i \leqslant m$ such that
 \[
  f = \sum_{i=1}^m h_i \circ g_i.
 \]
\end{lemma}

\begin{proof}
 If $v_+$ is a highest weight vector of $V_{(r-1) \cdot \rho}$ and $v$ is a weight vector of $V$ then $c_{V_{(r-1) \cdot \rho},V}(v_+ \otimes v)$ is proportional to $v \otimes v_+$ because
 \[
  \rho_{V_{(r-1) \cdot \rho}} \left( \prod_{k=1}^N E_{\beta_k}^{b_k} \right)(v_+) = 0
 \]
 for all integers $0 \leq b_1,\ldots,b_N < r$ whose sum is strictly positive. Furthermore, $c_{V,V_{(r-1) \cdot \rho}}(f(v) \otimes v_+)$ is proportional to $v_+ \otimes f(v)$ because $f$ is transparent in $\calC^H$. But now
 \[
  \left\{ \rho_{V_{(r-1) \cdot \rho}} \left( \prod_{k=1}^N F_{\beta_k}^{b_k} \right)(v_+) \Biggm| 0 \leq b_1,\ldots,b_N < r \right\}
 \]
 is a basis of $V_{(r-1) \cdot \rho}$ thanks to Proposition 34 of \cite{GP13}. This means that
 \[
  \rho_V \left( \prod_{k=1}^N E_{\beta_k}^{b_k} \right)(f(v)) = 0
 \]
 for every weight vector $v \in V$ and for all integers $0 \leq b_1,\ldots,b_N < r$ whose sum is strictly positive.

 Analogously, if $v_-$ is a lowest weight vector of $V_{(r-1) \cdot \rho}$ and $v$ is a weight vector of $V$ then  $c_{V,V_{(r-1) \cdot \rho}}(v \otimes v_-)$ is proportional to $v_- \otimes v$ because
 \[
  \rho_{V_{(r-1) \cdot \rho}} \left( \prod_{k=1}^N F_{\beta_k}^{b_k} \right)(v_-) = 0
 \]
 for all integers $0 \leq b_1,\ldots,b_N < r$ whose sum is strictly positive.  Furthermore, $c_{V_{(r-1) \cdot \rho},V}(v_- \otimes f(v))$ is proportional to $f(v) \otimes v_-$ because $f$ is transparent in $\calC^H$. But now
 \[
  \left\{ \rho_{V_{(r-1) \cdot \rho}} \left( \prod_{k=1}^N E_{\beta_k}^{b_k} \right)(v_-) \Biggm| 0 \leq b_1,\ldots,b_N < r \right\}
 \]
 is a basis of $V_{(r-1) \cdot \rho}$ thanks to Proposition 34 of \cite{GP13}. This means that
 \[
  \rho_V \left( \prod_{k=1}^N F_{\beta_k}^{b_k} \right)(f(v)) = 0
 \]
 for every weight vector $v \in V$ and for all integers $0 \leq b_1,\ldots,b_N < r$ whose sum is strictly positive.

 Now, since $K_i - K_i^{-1} = (q_i-q_i^{-1}) \cdot [E_i,F_i]$ for every integer $1 \leqslant i \leqslant n$, we get the equality $\rho_V(K_i)(f(v)) - \rho_V(K_i^{-1})(f(v)) = 0$ for every $v \in V$, which implies 
 \[
  \rho_V(K_i)^2(f(v)) = f(v).
 \]
 Then, if $f(v)$ is a weight vector of weight $\kappa$, this tells us that $2 \langle \kappa,\alpha_i \rangle \in r \Z$ for every integer $1 \leqslant i \leqslant n$. Since $r$ is odd and coprime with $d_i$ for every integer $1 \leqslant i \leqslant n$, this means precisely that $\kappa \in \PGr$. Therefore, each weight vector of $\im f$ determines a 1-dimensional submodule which is isomorphic to $\sigma(\kappa)$ for some $\kappa \in \PGr$. Since $\im f$ is a direct sum of its weight spaces, this means that
 \[
  \im f \cong \bigoplus_{i=1}^m \sigma(\kappa_i)
 \]
 for some integer $m \geqslant 1$ and some $\kappa_1, \ldots, \kappa_m \in \PGr$. Let $\pi_i \in \Hom_{\calC^H}(\im f,\sigma(\kappa_i))$ and $\iota_i \in \Hom_{\calC^H}(\sigma(\kappa_i),\im f)$ denote the corresponding projection and injection morphisms for every integer $1 \leqslant i \leqslant m$. We can factorize $f = \iota_f \circ \pi_f$ where $\pi_f \in \Hom_{\calC^H}(V,\im f)$ is naturally induced by $f$ and $\iota_f \in \Hom_{\calC^H}(\im f,V)$ denotes inclusion. Then the result follows by setting $g_i := \pi_i \circ \pi_f$ and $h_i := \iota_f \circ \iota_i$ for every integer $1 \leqslant i \leqslant m$.
\end{proof}

Let us complete $\{ 0 \} \subset \Lambda_R$ to a set
$\calH_r \subset \Lambda_R$ of representatives of equivalence classes
in $\Lambda_R / \PGr$. Similarly, let us choose a set
$I_{G \smallsetminus X} \subset \ddot{\frakh}^*$ of representatives of
equivalence classes in $G \smallsetminus X$. If for every
$\gamma \in I_{G \smallsetminus X}$ we define
$I_{[\gamma]} := \{ \gamma \} + \calH_r$, then
\[
 \Theta(\calC_{[\gamma]}) := \{ V_\mu \in \calC_{[\gamma]} \mid \mu \in I_{[\gamma]} \}
\]
is a set of representatives of $\PGr$-orbits of isomorphism classes of simple objects of $\calC_{[\gamma]}$. We have now everything in place to prove Theorem \ref{T:relative_modularity}.

\begin{figure}[hb]
  \centering
  \includegraphics{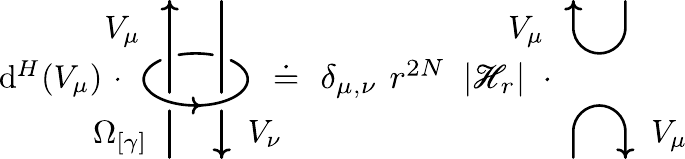}
  \caption{Relative modularity of $\calC^H$.}
  \label{F:cutting_ev-coev}
 \end{figure}

\begin{proof}[Proof of Theorem \ref{T:relative_modularity}]
 We know $\calC^H$ is a non-degenerate relative pre-modular category thanks to Theorem 7.2 of \cite{CGP14} and to Theorem 4 of \cite{GP18}. Therefore, we only need to prove that $\calC^H$ satisfies the relative modularity condition of Definition \ref{D:relative_modular}. We will do this by showing the skein equivalence of Figure \ref{F:cutting_ev-coev} for every $[\gamma] \in G \smallsetminus X$, for every $\mu \in I_{G \smallsetminus X}$, and for every $\nu \in I_{[\mu]}$, so let $f_{[\gamma],\mu,\nu}$ denote the morphism of $\calC^H$ obtained by applying the Reshetikhin-Turaev functor $F_{\calC^H}$ to the $\calC^H$-colored ribbon graph represented in the left hand part of Figure \ref{F:cutting_ev-coev}, ignoring the coefficient. Thanks to the handle slide property, we have the skein equivalence of Figure \ref{F:transparent_ev-coev},
 \begin{figure}[bt]
  \centering
  \includegraphics{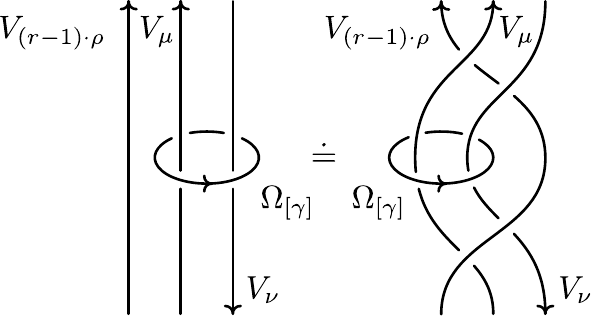}
  \caption{Transparency of $f_{[\gamma],\mu,\nu}$.}
  \label{F:transparent_ev-coev}
 \end{figure} 
 which means $f_{[\gamma],\mu,\nu}$ is transparent in $\calC^H_{[0]}$. Now, thanks to Lemma \ref{L:transparency}, we have
 \[
  f_{[\gamma],\mu,\nu} = \sum_{i=1}^m h_{[\gamma],\mu,\nu,i} \circ g_{[\gamma],\mu,\nu,i}
 \]
 with $g_{[\gamma],\mu,\nu,i} \in \Hom_{\calC^H}(V_{\mu} \otimes V_{\nu}^*,\sigma(\kappa_i))$, with $h_{[\gamma],\mu,\nu,i} \in \Hom_{\calC^H}(\sigma(\kappa_i),V_{\mu} \otimes V_{\nu}^*)$, and with $\kappa_i \in \PGr$ for every integer $1 \leq i \leq m$. But now $h_{[\gamma],\mu,\nu,i} = 0$ unless $\mu =\nu$ and $\kappa_i = 0$, because
 \[
  \Hom_{\calC^H} (V_{\mu} \otimes V_{\nu}^*,\sigma(\kappa_i)) \cong \Hom_{\calC^H} (V_{\mu},V_{\nu} \otimes \sigma(\kappa_i))
 \]
 and because $\calH_r$ is a set of representatives of equivalence classes in $\Lambda_R / \PGr$. This means that $f_{[\gamma],\mu,\mu}$ factors through the tensor unit $\one$. But now, since $V_{\mu}$ is simple, both $\Hom_{\calC^H} (V_{\mu} \otimes V_{\mu}^*,\one)$ and $\Hom_{\calC^H} (\one, V_{\mu} \otimes V_{\mu}^*)$ are 1-dimensional. This means $f_{[\gamma],\mu,\mu}$ is a scalar multiple of $\lcoev_{V_{\mu}} \circ \rev_{V_{\mu}}$. In order to compute the proportionality coefficient let us compare the m-traces of $\lcoev_{V_{\mu}} \circ \rev_{V_{\mu}}$ and of $f_{[\gamma],\mu,\mu}$. The first m-trace is easily seen to be
 \[
  \rmt^H_{V_{\mu} \otimes V_{\mu}^*}(\lcoev_{V_{\mu}} \circ \rev_{V_{\mu}}) = \rmt^H_{V_{\mu}}(\id_{V_{\mu}}) = \rmd^H(V_{\mu}).
 \]
 Indeed, this follows immediately from the partial trace property of $\rmt^H$. On the other hand, if
 \[
  f_{[\gamma],\mu,\mu} = \sum_{\nu \in \{ \gamma \} + \calH_r} \rmd^H(V_{\nu}) \cdot f_{\nu,\mu,\mu},
 \]
 where $f_{\nu,\mu,\mu}$ is obtained from $f_{[\gamma],\mu,\mu}$ by replacing the label $\Omega_{[\gamma]}$ of the meridian with $V_{\nu}$, then, the second m-trace is given by
 \begin{align*}
  \rmt^H_{V_{\mu} \otimes V_{\mu}^*}(f_{[\gamma],\mu,\mu}) 
  &= \sum_{\nu \in \{ \gamma \} + \calH_r} \rmd^H(V_{\nu}) \rmt^H_{V_{\mu} \otimes V_{\mu}^*}(f_{\nu,\mu,\mu}) \\
  &= \sum_{\nu \in \{ \gamma \} + \calH_r} \rmd^H(V_{\nu}) \rmt^H_{V_{\nu}}(f_{\mu,\nu}^- \circ f_{\mu,\nu}^+) \\
  &= \sum_{\nu \in \{ \gamma \} + \calH_r} \rmd^H(V_{\nu}) \rmd^H(V_{\nu})^{-1} \rmt^H_{V_{\nu}}(f_{\mu,\nu}^-) \rmt^H_{V_{\nu}}(f_{\mu,\nu}^+) \\
  &= r^{2N} \left| \calH_r \right|,
 \end{align*}
 where the morphisms $f_{\mu,\nu}^-$ and $f_{\mu,\nu}^+$ are represented in Figure \ref{F:long_Hopf_links}. Indeed, the second equality follows from both the cyclicity and the partial trace properties of $\rmt^H$, using isotopy, the third equality follows from the fact that $V_\nu$ is simple, and the fourth equality follows from the very last equation of Subsection \ref{Subs:unrolled_quantum_groups}.
\end{proof}

\begin{corollary}
 The CGP invariant $\rmN_{\calC^H}$ extends to a $\PGr$-graded TQFT
 \[
  \bbV_{\calC^H}^\PGr : \adCob_{\calC^H}^G \rightarrow \Vect_{\C}^{\PGr}.
 \]
\end{corollary}

\subsection{Projective generators and forgetful functor}\label{Subs:projective_generators}

In this subsection we prove some key technical results which will be later used for the proof Theorem \ref{T:main_result}. We say an object $P$ of $\bar{\calC}$ is a \textit{projective generator of $\bar{\calC}$} if for every object $V$ of $\Proj(\bar{\calC})$ there exist some integer $m$ and some morphisms $f_i \in \Hom_{\bar{\calC}}(V,P)$ and $g_i \in \Hom_{\bar{\calC}}(P,V)$ for every integer $1 \leqslant i \leqslant m$ such that
\[
 \id_V = \sum_{i=1}^m g_i \circ f_i.
\]
Remark that a natural choice for a projective generator of $\bar{\calC}$ is the regular representation $\bar{U}$ of $\bar{U}_q \frakg$. Analogously, we say an object $P$ of $\calC^H_{[0]}$ is a \textit{relative projective generator of $\calC^H_{[0]}$} if for every object $V$ of $\Proj(\calC^H_{[0]})$ there exist some integer $m$, some $\kappa_i \in \PGr$, and some morphisms $f_i \in \Hom_{\calC^H}(V,P \otimes \sigma(\kappa_i))$ and $g_i \in \Hom_{\calC^H}(P \otimes \sigma(\kappa_i),V)$ for every integer $1 \leqslant i \leqslant m$ such that
\[
 \id_V = \sum_{i=1}^m g_i \circ f_i.
\]
Next, we apply the following result to $\calC^H$.

\begin{proposition}\label{P:IndecProj}
 Let $\cat$ be an abelian relative pre-modular $G$-category over an algebraically closed field $\kk$. Then $\cat$ has enough projectives, and every indecomposable projective object is a projective cover of a simple object.
\end{proposition}

\begin{proof}
 Choose some $g \in G \smallsetminus X$ and some simple object $V \in \calC_g$. By definition $\cat_g$ is semisimple, and so $V$ is projective and has epic evaluation. Then for every $W \in \cat$ the morphism
 \[
  \id_{W} \otimes \rev_{V} \in \Hom_{\cat}(W \otimes V \otimes V^*,W).
 \]
 is an epimorphism from a projective object to $W$. Thus $\cat$ has enough projectives. Next, for the second statement, let $P$ be an indecomposable projective object of $\cat_h$ for some $h \in G$. Since $X$ is small in $G$, there exists some $g \in G$ satisfying $g+h \in G \smallsetminus X$. Then, let $V$ be a simple object of $\cat_g$. The algebra $\End_\cat(P)$ embeds into $\End_\cat (P \otimes V) \subset \Mat_N(\kk)$ with $N$ bounded by the number of simple summands of $P \otimes V$. Then 
 Fitting's Lemma applies to $\End_\cat(P)$, i.e. its elements are either nilpotent or  isomorphisms. We can now prove that $P$ has a unique maximal proper subobject. To see this, let us suppose by contradiction $M_1$ and $M_2$ are distinct maximal proper subobjects. Then we have an epimorphism $p \in \Hom_{\calC^H}(M_1 \oplus M_2,P)$ which admits a section $(f_1,f_2) \in \Hom_{\calC^H}(P,M_1 \oplus M_2)$, because $P$ is projective. If $i_j \in \Hom_{\calC^H}(M_j,P)$ denotes the inclusion and $\tilde{f}_j := i_j \circ f_j$, then $\tilde{f}_1 + \tilde{f}_2 = \id_P$. Since their sum is the identity of $P$, they commute, and $\tilde{f}_1$ and $\tilde{f}_2$ cannot be simultaneously nilpotent. Thanks to Fitting's Lemma, $\tilde{f}_j$ is an isomorphism for some $j$. This implies that $P$ is a retract of $M_j$, and thus $P$ is a proper subobject of $P$. Then $P \otimes V$ is a proper subobject of itself, but $P \otimes V$ is also semisimple, which is a contradiction. 
\end{proof}

We will denote by $P_\mu$ a projective cover of $V_\mu$ for any $\mu \in \frakh^*$.
Thanks to Proposition \ref{P:IndecProj}, if $\calH_r$ denotes the set of representatives of equivalence classes in $\Lambda_R / \PGr$ of Figure \ref{F:cutting_ev-coev}, then
\[
 \bfP := \bigoplus_{\mathclap{\mu \in \calH_r}} P_{\mu}
\]
is by construction a relative projective generator of $\calC^H_{[0]}$.

\begin{lemma}\label{L:projective_generator}
 $\dim_{\C} \left( \Hom_{\calC^H}(\bfP,\sigma(\kappa)) \right) = \dim_{\C} \left( \Hom_{\calC^H}(\sigma(\kappa),\bfP) \right) = \delta_{0,\kappa}$ for every $\kappa \in \PGr$.
\end{lemma}

\begin{proof}
 Since the vector space of $U^H_q\frakg$-module morphisms from a projective indecomposable weight $U^H_q\frakg$-module to its unique simple quotient is 1-dimensional, and since $\{ \sigma(\kappa) \in \calC^H_{[0]} \mid \kappa \in \PGr \}$ is the $\PGr$-orbit of the tensor unit $\one$,
 we have
 \[
  \dim_{\C} \left( \Hom_{\calC^H}(\bfP,\sigma(\kappa)) \right) = \delta_{0,\kappa}.
 \]
 Furthermore, since the dual of an indecomposable projective weight $U^H_q\frakg$-module is also indecomposable and projective, see for instance Proposition
6.1.3 of \cite{EGNO15}, Proposition \ref{P:IndecProj} implies that each indecomposable projective weight $U^H_q\frakg$-module $P$ has an unique simple submodule which is dual to the simple quotient of $P^*$. Since $\calC^H$ is unimodular, see \cite{GP13} and Theorem 3.1.3 of \cite{GKP13}, $P_0$ is the unique indecomposable projective weight $U^H_q\frakg$-module that constains the tensor unit $\one = \sigma(0)$ as a submodule, and similarly $P_0 \otimes \sigma(\kappa)$ is the unique indecomposable projective weight $U^H_q\frakg$-module that contains $\sigma(\kappa)$ as a submodule. Hence we have
 \[
   \dim_{\C} \left( \Hom_{\calC^H}(\sigma(\kappa),\bfP) \right)= \delta_{0,\kappa}. \qedhere
 \]
\end{proof}

The following result establishes a cutting property for Kirby meridians which is analogous to Lemma 3.6 of \cite{DGP18}.

\begin{lemma}\label{L:cutting_bfP}
 There exist generators $\bfepsilon \in \Hom_{\calC^H}(\bfP,\one)$ and $\bfLambda \in \Hom_{\calC^H}(\one,\bfP)$ satisfying $\rmt^H_{\bfP}(\bfLambda \circ \bfepsilon) = 1$ which realize the skein equivalence of Figure \ref{F:cutting_bfP}.
\end{lemma}

\begin{figure}[ht]
 \centering
 \includegraphics{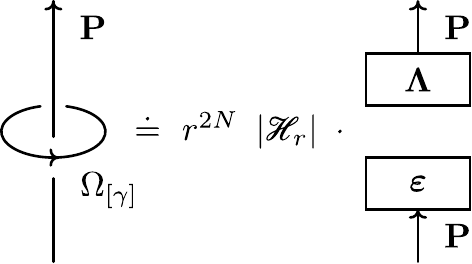}
 \caption{Cutting property for Kirby-colored meridians.}
 \label{F:cutting_bfP}
\end{figure}

\begin{proof}
 Thanks to Lemma \ref{L:projective_generator}, and thanks to the non-degeneracy of $\rmt^H$, the composition of a non-trivial morphism of $\Hom_{\calC^H}(\bfP,\one)$ with a non-trivial morphism of $\Hom_{\calC^H}(\one,\bfP)$ has non-zero m-trace. Therefore, let us fix a pair of generators $\bfepsilon \in \Hom_{\calC^H}(\bfP,\one)$ and $\bfLambda \in \Hom_{\calC^H}(\one,\bfP)$ satisfying $\rmt^H_{\bfP}(\bfLambda \circ \bfepsilon) = 1$, and let us prove they realize the skein equivalence of Figure \ref{F:cutting_bfP}. If $h_{[\gamma]}$ is the morphism of $\calC^H$ obtained by applying the Reshetikhin-Turaev functor $F_{\calC^H}$ to the $\calC^H$-colored ribbon graph represented in the left hand part of Figure \ref{F:cutting_bfP}, the handle slide property yields the skein equivalence represented in Figure \ref{F:transparent_bfP}.
 \begin{figure}[b]
  \centering
  \includegraphics{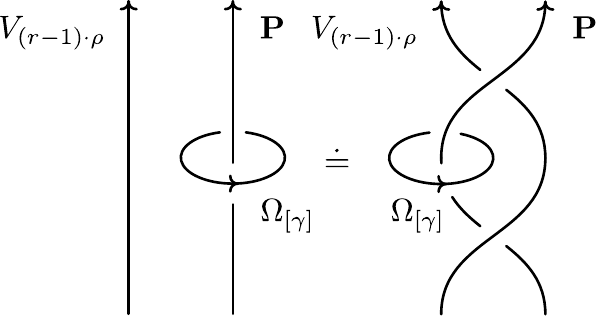}
  \caption{Transparency of $h_{[\gamma]}$.}
  \label{F:transparent_bfP}
 \end{figure}
 This means that, thanks to Lemmas \ref{L:transparency} and \ref{L:projective_generator}, the morphism $h_{[\gamma]}$ factors through the tensor unit $\one$. But then, since $\Hom_{\calC^H} (\bfP,\one)$ and $\Hom_{\calC^H} (\one,\bfP)$ are both 1-dimensional, we must have $h_{[\gamma]} = \alpha \cdot \bfLambda \circ \bfepsilon$ for some $\alpha \in \C$. Then, let us show $\alpha = r^{2N} \left| \calH_r \right|$. Proposition \ref{P:IndecProj} implies $\Hom_{\calC^H}(P_\mu,\one)$ is $\delta_{\mu,0}$-dimensional for every $\mu \in \calH_r$. 
 Moreover, for every $\mu \in \ddot{\frakh}^*$, the projective cover $P_0$ of $V_0 = \one$ is a direct summand of $V_\mu \otimes V_\mu^*$ of multiplicity 1. In other words, we have
 \[
  \id_{V_\mu \otimes V_\mu^*} = \sum_{i=0}^m \iota_{\mu,i} \circ \pi_{\mu,i}
 \]
 for some weights $\mu_1,\ldots,\mu_m \in \Lambda_R$ satisfying $\mu_0 = 0$ and $\mu_i \neq 0$ for all integers $1 \leq i \leq m$, and for some $U^H_q \frakg$-module morphisms $\pi_{\mu,i} \in \Hom_{\calC^H}(V_\mu \otimes V_\mu^*,P_{\mu_i})$ and $\iota_{\mu,i} \in \Hom_{\calC^H}(P_{\mu_i},V_\mu \otimes V_\mu^*)$. Remark that this implies
 \[
  \lcoev_{V_\mu} \circ \rev_{V_\mu} = \lcoev_{V_\mu} \circ \rev_{V_\mu} \circ \iota_{\mu,0} \circ \pi_{\mu,0}.
 \]
 Therefore, let us fix some $\mu \in I_{G \smallsetminus X}$, and let us consider the $U^H_q \frakg$-module morphisms $f_\mu := \iota_{\mu,0} \circ \pi_0 \in \Hom_{\calC^H}(\bfP,V_\mu \otimes V_\mu^*)$ and $g_\mu := \iota_0 \circ \pi_{\mu,0} \in \Hom_{\calC^H}(V_\mu \otimes V_\mu^*,\bfP)$ determined by the projection $\pi_0 \in \Hom_{\calC^H}(\bfP,P_0)$ and the injection $\iota_0 \in \Hom_{\calC^H}(P_0,\bfP)$. Thanks to Figure \ref{F:cutting_ev-coev}, we have the skein equivalence of Figure \ref{F:cutting_black}.
 \begin{figure}[t]
  \centering
  \includegraphics{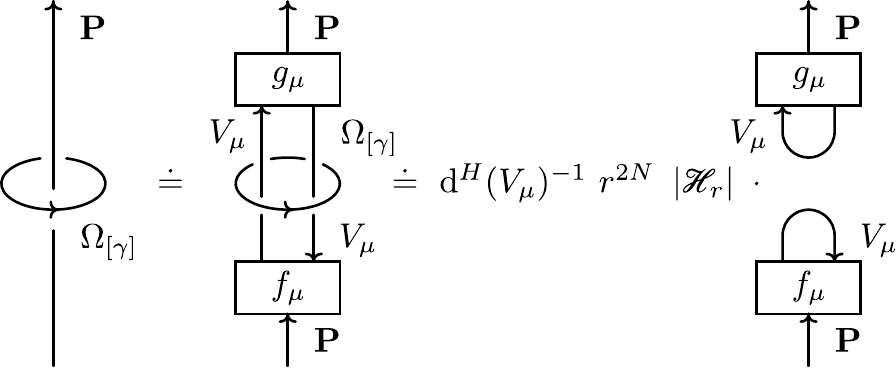}
  \caption{Cutting a $\bfP$-colored edge with a Kirby-colored meridian.}
  \label{F:cutting_black}
 \end{figure}
 Then, let us compute the m-trace of $h_{[\gamma]}$ in two different ways. On one hand, we have
 \begin{align*}
  \rmt^H_\bfP(h_{[\gamma]}) &= \alpha \rmt^H_\bfP(\bfLambda \circ \bfepsilon) = \alpha.
 \end{align*}
 On the other hand, we have
 \begin{align*}
  \rmt^H_\bfP(h_{[\gamma]}) &= \rmd^H(V_\mu)^{-1} r^{2N} \left| \calH_r \right| \rmt^H_\bfP( g_\mu \circ \lcoev_{V_\mu} \circ \rev_{V_\mu} \circ f_\mu) \\
  &= \rmd^H(V_\mu)^{-1} r^{2N} \left| \calH_r \right| \rmt^H_{V_\mu \otimes V_\mu^*}(\lcoev_{V_\mu} \circ \rev_{V_\mu} \circ f_\mu \circ g_\mu) \\
  &= \rmd^H(V_\mu)^{-1} r^{2N} \left| \calH_r \right| \rmt^H_{V_\mu \otimes V_\mu^*}(\lcoev_{V_\mu} \circ \rev_{V_\mu} \circ \iota_{\mu,0} \circ \pi_{\mu,0}) \\
  &= \rmd^H(V_\mu)^{-1} r^{2N} \left| \calH_r \right| \rmt^H_{V_\mu \otimes V_\mu^*}(\lcoev_{V_\mu} \circ \rev_{V_\mu}) \\
  &= \rmd^H(V_\mu)^{-1} r^{2N} \left| \calH_r \right| \rmt^H_{V_\mu}(\id_{V_\mu}) \\
  &= r^{2N} \left| \calH_r \right|,
 \end{align*}
 where the second and fifth equalities follow from the cyclicity and the partial trace properties of $\rmt^H$ respectively. 
\end{proof}

Let us consider the forgetful functor $\Phi_{\calC} : \calC^H_{[0]} \rightarrow \bar{\calC}$ which forgets the action of $H_i$ for all integers $1 \leqslant i \leqslant n$. If $V$ is an object of $\calC^H_{[0]}$ we denote with $\bar{V}$ its image under $\Phi_{\calC}$, and if $f$ is a morphism of $\calC^H_{[0]}$ we denote with $\bar{f}$ its image under $\Phi_{\calC}$. This induces a ribbon functor $\Phi_{\calR} : \calR_{\calC^H_{[0]}} \rightarrow \calR_{\bar{\calC}}$ from the category of $\calC^H_{[0]}$-colored ribbon graphs to the category of $\bar{\calC}$-colored ribbon graphs. If $T$ is a morphism of $\calR_{\calC^H_{[0]}}$ we denote with $\bar{T}$ its image under $\Phi_{\calR}$.

\begin{lemma}\label{L:forgetful_functors}
 The forgetful functor $\Phi_{\calC} : \calC^H_{[0]} \rightarrow \bar{\calC}$ is ribbon, and it satisfies
 \[
  \Phi_{\calC} \circ F_{\calC^H_{[0]}} = F_{\bar{\calC}} \circ \Phi_{\calR}.
 \]
\end{lemma}

\begin{proof}
 First of all, remark that $K_{2 \cdot \rho}^{1-r} = K_{2 \cdot \rho}$ in $\bar{U}_q \frakg$. Then the result follows immediately from the equality 
 \[
  R^H_{0,V,V'} = (\rho_V  \otimes \rho_{V'})(\bar{R}_0)
 \]
 for all $V, V' \in \calC^H_{[0]}$, where $R^H_{0,V,V'} : V \otimes V' \rightarrow V \otimes V'$ is defined in Subsection \ref{Subs:unrolled_quantum_groups}, and where $\bar{R}_0 \in \bar{U}_q \frakh \otimes \bar{U}_q \frakh$ is defined in Subsection \ref{Subs:small_quantum_groups}. To show the claim, remark
 \[
  \sum_{\mu' \in \Lambda_R / Z} q^{\langle \mu,\mu' \rangle} = \left| \Lambda_R / Z \right| \delta_{\mu,0}
 \]
 for every $\mu \in \Lambda_R$. This means that
 \begin{align*}
  (\rho_V \otimes \rho_{V'})(\bar{R}_0)(v \otimes v')
  &= \frac{1}{\left| \Lambda_R / Z \right|} \cdot \sum_{\smash{\mu,\mu' \in \Lambda_R / Z}} q^{- \langle \mu,\mu' \rangle} \cdot \rho_V(K_{\mu})(v) \otimes \rho_{V'}(K_{\mu'})(v') \\
  &= \frac{1}{\left| \Lambda_R / Z \right|} \cdot  \sum_{\smash{\mu,\mu' \in \Lambda_R / Z}} q^{\langle \mu,\nu \rangle + \langle \mu',\nu' \rangle - \langle \mu,\mu' \rangle} \cdot v \otimes v' \\
  &= \sum_{\mu \in \Lambda_R / Z} \delta_{\mu,\nu'} q^{\langle \mu,\nu \rangle} \cdot v \otimes v' \\
  &= q^{\langle \nu,\nu' \rangle} \cdot v \otimes v'
 \end{align*}
 for all $v \in V$, $v' \in V'$ satisfying
 \[
  \rho_V(H_i)(v) = \nu(H_i) \cdot v, \quad \rho_{V'}(H_i)(v') = \nu'(H_i) \cdot v'
 \]
 for every integer $1 \leqslant i \leqslant n$.
\end{proof}

The forgetful functor $\Phi_{\calC}$ preserves the property of being projective.

\begin{lemma}\label{L:projectives}
 If $P$ is a projective object of $\calC^H_{[0]}$, then $\bar{P}$ is a projective object of $\bar{\calC}$.
\end{lemma}

\begin{proof}
 The typical $U^H_q\frakg$-module $V_{(r-1) \cdot \rho}$ introduced in Subsection \ref{Subs:unrolled_quantum_groups} generates $\Proj(\calC^H_{[0]})$ thanks to Lemma 17 of \cite{GPV13}. Then $P$ must be a direct summand of a tensor product $V_{(r-1) \cdot \rho} \otimes W$ for some $W \in \calC^H_{[0]}$. Now the proof of Lemma 7.1 of \cite{CGP14} can be repeated to show the image $\bar{V}_{(r-1) \cdot \rho}$ of $V_{(r-1) \cdot \rho}$ under the forgetful functor $\Phi_{\calC}$ is projective. This means $\bar{V}_{(r-1) \cdot \rho}$ generates $\Proj(\bar{\calC})$, and thus $\bar{P}$, which is a direct summand of $\bar{V}_{(r-1) \cdot \rho} \otimes \bar{W}$, is projective.
\end{proof}

In particular, the image $\bar{\bfP}$ of the relative projective generator $\bfP$ of $\calC^H_{[0]}$ defined in Subsection \ref{Subs:projective_generators} is projective.

\begin{lemma}\label{L:dimOne}
 $\dim_{\C} \left( \Hom_{\bar{\calC}}(\one,\bar{\bfP}) \right) = \dim_{\C} \left( \Hom_{\bar{\calC}}(\bar{\bfP},\one) \right) = 1$.
\end{lemma}

\begin{proof}
 Let $V$ be a weight $\smash{U^H_q \frakg}$-module in $\smash{\calC^H_{[0]}}$, and remark that its image $\bar{V}$ under $\Phi_{\calC}$ coincides with $V$ as a vector space. Let us consider the space $\bar{V}^{\bar{U}_q \frakg}$ of $\bar{U}_q \frakg$-invariants vectors of $\bar{V}$. Remark that $\Hom_{\bar{\calC}}(\one,\bar{V})$ is naturally isomorphic to $\bar{V}^{\bar{U}_q \frakg}$, simply by identifying every morphism $f \in \Hom_{\bar{\calC}}(\one,\bar{V})$ with the image $f(1) \in \bar{V}^{\bar{U}_q \frakg}$. We claim the subspace $V^{\bar{U}_q \frakg}$ of $V$ formed by vectors of $\bar{V}^{\bar{U}_q \frakg}$ is a $U^H_q \frakg$-submodule of $V$. Indeed, this follows from the commutation relations satisfied by the additional generators $H_1, \ldots, H_n$. But now remark that every weight vector of $V^{\bar{U}_q \frakg}$ with respect to the action of $U^H_q \frakg$ determines a split 1-dimensional submodule of $V^{\bar{U}_q \frakg}$ which is isomorphic to $\sigma(\kappa)$ for some $\kappa \in \PGr$, as all 1-dimensional weight $U^H_q \frakg$-modules in $\calC^H_{[0]}$ are. This means that
 \[
  \smash{V^{\bar{U}_q \frakg}} \cong \bigoplus_{i=1}^{\smash{\dim_{\C}(\bar{V}^{\bar{U}_q \frakg})}} \sigma(\kappa_i)
 \]
 with $\kappa_i \in \PGr$ for every integer $1 \leqslant i \leqslant \smash{\dim_{\C}(\bar{V}^{\bar{U}_q \frakg})}$. Thus we get 
 \[
  \dim_{\C}(\bar{V}^{\bar{U}_q \frakg}) \leq \dim_{\C} \left( \bigoplus_{\kappa \in \PGr} \Hom_{\calC^H} \left( \sigma(\kappa),V \right) \right),
\]
and the converse inequality follows from the equality $\Phi_{\calC}(\sigma(\kappa)) = \one$ for every $\kappa \in \PGr$.

 First, let us consider $V = \bfP$. Thanks to Lemma \ref{L:projective_generator}, we have 
 \[
  \dim_{\C} \left( \bigoplus_{\kappa \in \PGr} \Hom_{\calC^H} \left( \sigma(\kappa),\bfP \right) \right) = 1.
 \]  
 Thus, the space $\bar{\bfP}^{\bar{U}_q \frakg}$ is 1-dimensional. This means 
 \[
  \dim_{\C} \left( \Hom_{\bar{\calC}}(\one,\bar{\bfP}) \right) = 1.
 \]
 Next, let us consider $V = \bfP^*$. Thanks to Lemma \ref{L:projective_generator}, we have 
 \[
  \dim_{\C} \left( \bigoplus_{\kappa \in \PGr} \Hom_{\calC^H} \left( \sigma(\kappa),\bfP^* \right) \right) = \dim_{\C} \left( \bigoplus_{\kappa \in \PGr} \Hom_{\calC^H} \left( \bfP,\sigma(-\kappa) \right) \right) = 1.
 \]  
 Thus, the space $(\bar{\bfP}^*)^{\bar{U}_q \frakg}$ is 1-dimensional. This means
 \[
  \dim_{\C} \left( \Hom_{\bar{\calC}}(\bar{\bfP},\one) \right) = \dim_{\C} \left( \Hom_{\bar{\calC}}(\one,\bar{\bfP}^*) \right) = 1. \qedhere
 \]
\end{proof}

The last result we will need requires an additional hypothesis.

\begin{proposition}\label{P:barP_proj_gen}
 If $\gcd(r,\det(A)) = 1$, then $\bar{\bfP}$ is a projective generator of $\bar{\calC}$.
\end{proposition}

\begin{proof}
%
  Let us show that every indecomposable projective
  $\bar{U}_q \frakg$-module $P$ is isomorphic to a direct summand of
  $\bar{\bfP}$. Let $V$ be the unique simple quotient of $P$. Its
  highest weight is a ring homomorphism
  $\varphi : \bar{U}_q \frakh \to \C$ assigning to each $K_i$ the root
  of unity by which it acts on the highest weight vector of $V$. Since
  $r$ is coprime with $d_1, \ldots, d_n$ defined in Appendix
  \ref{A:quantum_groups}, there exists some $\omega \in \Lambda_W$
  satisfying $\varphi(K_i)=q^{\brk{\omega,\alpha_i}}$ for every
  integer $1 \leq i \leq n$. Furthermore, since $r$ is coprime also
  with $\det(A)$, the matrix $(d_ia_{ij})_{1 \leq i,j \leq n}$ is
  invertible modulo $r$. In particular, there exist
  $\tilde{\omega}_i \in \Lambda_R$ such that
  $\brk{\tilde{\omega}_i,\alpha_j} \equiv \delta_{ij}$ modulo $r$ for
  all integers $1 \leq i,j \leq n$. Then, if we set
 \[
  \tilde{\omega} := \sum_{i=1}^n \langle \omega,\alpha_i \rangle \cdot \tilde{\omega}_i \in \Lambda_R,
 \]
 the simple weight $U^H_q \frakg$-module $V_{\tilde{\omega}} \in \calC^H_{[0]}$ satisfies $\Phi_{\calC}(V_{\tilde{\omega}}) \cong V$,
 because $\Phi_{\calC}(V_{\tilde{\omega}})$ and $V$ have the same highest weight. Moreover, if $P_{\tilde{\omega}}$ is a projective cover of $V_{\tilde{\omega}}$, then $P_{\tilde{\omega}} \otimes \sigma(\kappa)$ is isomorphic to a direct summand of $\bfP$ for some $\kappa \in \PGr$. But now, since $\Phi_{\calC}(P_{\tilde{\omega}} \otimes \sigma(\kappa))$ is projective, it must contain $P$ as a direct summand, so this proves our statement. 
%
\end{proof}

\section{Equality of 3-manifold invariants}

The goal for this section is to prove Theorem \ref{T:main_result}. We will use as a key ingredient the fact that meridians labeled with Kirby colors have the cutting property with respect to the relative projective generator $\bfP$ of $\smash{\calC^H_{[0]}}$, while red meridians labeled with the regular representation have the cutting property with respect to its image $\bar{\bfP}$ in $\bar{\calC}$. The proof will require a comparison of all the ingredients that correspond to each other in the two theories. Since it is part of the hypotheses of Theorem  \ref{T:main_result}, we will suppose throughout this section that $\gcd(r,\det(A)) = 1$.

\subsection{Stabilized surgery presentations}\label{Subs:stabilized_surgery_presentations}

In this subsection we introduce special surgery presentations of admissible decorated closed 3-manifolds which are tailored for the comparison between the CGP and the renormalized Hennings invariants. We start with a preliminary comment about our notation: we will always denote with $F_{\bar{\calC}}$ the restriction of $F_\lambda$ to $\calC$-colored blue ribbon graphs, and similarly we will always denote with $F'_{\bar{\calC}}$ the restriction of $F'_\lambda$ to admissible closed $\bar{\calC}$-colored blue ribbon graphs, in order to stress the absence of red edges. Next, we need to compare the m-trace $\rmt^H$ on $\Proj(\calC^H)$ with the m-trace $\bar{\rmt}$ on $\Proj(\bar{\calC})$.

\begin{remark}\label{R:comparison_traces}
Both $\rmt^H$ and $\bar{\rmt}$ are unique up to scalar, but the chosen normalizations do not agree, as they are determined by the conditions
\[
 \rmt^H_{\bfP}(\bfLambda \circ \bfepsilon) = 1, \quad \bar{\rmt}_{\bar{U}}(\Lambda \circ \varepsilon) = 1
\]
respectively, where $\bfP$ is the relative projective generator of $\smash{\calC^H_{[0]}}$ introduced in Subsection \ref{Subs:projective_generators}, where $\bfepsilon \in \Hom_{\calC^H}(\bfP,\one)$ and $\bfLambda \in \Hom_{\calC^H}(\one,\bfP)$ are the morphisms introduced in Lemma \ref{L:cutting_bfP}, where $\bar{U}$ is the regular representation of $\bar{U}_q \frakg$, and where $\varepsilon$ and $\Lambda$ are the counit and the cointegral respectively.
Nevertheless, $\bar{\rmt} \circ \Phi_{\calC}$ clearly defines an m-trace
on $\Proj(\calC^H_{[0]})$: indeed, it obviously satisfies the cyclicity property, and the partial trace property is just a consequence of the fact that partial
traces commute with the ribbon functor $\Phi_{\calC}$ (see also Corollary 2.8 of
\cite{FG18} for a similar statment in the setting of
finite dimensional Hopf algebras).
Therefore, 
there exists a non-zero coefficient $\alpha \in \C^*$ such that
\[
 \restr{\rmt^H}{\calC^H_{[0]}} = \alpha \cdot \bar{\rmt} \circ \Phi_{\calC},
\]
meaning that $\rmt^H_V(f) = \alpha \bar{\rmt}_{\bar{V}}(\bar{f})$ for every $V \in \Proj(\calC^H_{[0]})$ and every $f \in \End_{\calC^H}(V)$. 
\end{remark}

\begin{lemma}\label{L:comparison_of_renormalized_graph_invariants}
 The renormalized invariants $F'_{\calC^H_{[0]}}$ and $F'_{\bar{\calC}}$ satisfy
 \[
  F'_{\calC^H_{[0]}} = \alpha \cdot F'_{\bar{\calC}} \circ \Phi_{\calR},
 \]
 meaning that $F'_{\smash{\calC^H_{[0]}}}(T) = \alpha F'_{\bar{\calC}}(\bar{T})$ for every closed admissible $\calC^H_{[0]}$-colored ribbon graph $T$, where $\alpha$ is the coefficient introduced in Remark \ref{R:comparison_traces}.
\end{lemma}

\begin{proof}
 If a closed admissible $\smash{\calC^H_{[0]}}$-colored ribbon graph $T$ admits a projective edge of color $V$, and if $T_V$ is a cutting presentation of $T$,
 then we have
 \[
  \rmt^H_V(F_{\calC^H_{[0]}}(T_V)) = \alpha \bar{\rmt}_{\bar{V}}(\Phi_{\calC}(F_{\calC^H_{[0]}}(T_V))) = \alpha \bar{\rmt}_{\bar{V}}(F_{\bar{\calC}}(\Phi_{\calR}(T_V))),
 \]
 where the second equality follows from Lemma \ref{L:forgetful_functors}.
\end{proof}

Now, let us recall the formulas defining the CGP and the renormalized Hennings invariants in the setting of Theorem \ref{T:main_result}. If $(M,T,0,0)$ is a closed connected morphism of $\adCob_{\calC^H}^G$, if $L = L_1 \cup \ldots \cup L_{\ell} \subset S^3$ is a surgery presentation of $M$, and if we replace $(T,0)$ with some $(\tilde{T},\tilde{\omega})$ obtained by projective stabilization of sufficiently generic index ensuring $L$ becomes computable, as explained in Subsection \ref{Subs:CGP_invariants} and, in greater detail, in Section 3.2 of \cite{D17}, then we have
\begin{equation*}\label{E:FormulaN00}
 \rmN_{\calC^H}(M,T,0,0) = \calD_{\Omega}^{-1-\ell} \delta_{\Omega}^{-\sigma(L)} F'_{\calC^H}(L \cup \tilde{T}).
\end{equation*}
On the other hand, if $(M,\bar{T},0)$ is the closed connected morphism of $\adCob_{\bar{\calC}}$ obtained by applying the functor $\Phi_{\calR}$ to $T \subset M$, then we have
\begin{equation*}\label{E:FormulaH'0}
 \rmH'_{\bar{\calC}}(M,\bar{T},0) = \calD_{\lambda}^{-1-\ell} \delta_{\lambda}^{-\sigma(L)} F'_{\lambda}(L \cup \bar{T}).
\end{equation*}
The surgery presentation $L$ is used in different ways by the two constructions. In the first case, $L$ is labeled with Kirby colors as prescribed by $\tilde{\omega}$, and thus it is not a morphism in the domain of $\Phi_{\calR}$. In the second case, $L$ is taken to be red, and thus it is not a morphism in the image of $\Phi_{\calR}$. In order to compare the two formulas, we introduce special morphisms of $\bar{\calC}$ which encode these two different procedures.

First, for all weights $\mu,\nu \in \ddot{\frakh}^*$ satisfying $[\mu] = [\nu] \in G \smallsetminus X$, the tensor product $W_{\mu,\nu} := V_{\mu} \otimes V_{\nu}^*$ is an object of $\smash{\calC^H_{[0]}}$. Therefore, since $\bar{U}$ is a projective generator of $\bar{\calC}$, we can fix a decomposition
\[
 \id_{\bar{W}_{\mu,\nu}} = \sum_{i=1}^{m_{\mu,\nu}} g_{\mu,\nu,i} \circ f_{\mu,\nu,i}
\]
for some morphisms $f_{\mu,\nu,i} \in \Hom_{\bar{\calC}}(\bar{W}_{\mu,\nu},\bar{U})$ and $g_{\mu,\nu,i} \in \Hom_{\bar{\calC}}(\bar{U},\bar{W}_{\mu,\nu})$. Let us also set
\[
 d_{\mu,\nu} := (p_{\mu,\nu} \otimes \id_{W_{\mu,\nu}}) \circ (\id_{V_{\nu}} \otimes \rcoev_{V_{\mu}} \otimes \id_{V_{\nu}^*}) \circ s_{\nu} \in \Hom_{\calC^H}(\bfP,W_{\mu,\nu}^* \otimes W_{\mu,\nu}),
\]
where $s_{\nu} \in \Hom_{\calC^H}(\bfP,W_{\nu,\nu})$ is a morphism satisfying $\rev_{V_{\nu}} \circ s_{\nu} = \bfepsilon$, and where $p_{\mu,\nu} \in \Hom_{\calC^H}(V_{\nu} \otimes V_{\mu}^*,(V_{\mu} \otimes V_{\nu}^*)^*)$ is the isomorphism coming from the pivotal structure of $\calC^H$. Now, let us fix once and for all a weight $\gamma \in \ddot{\frakh}^*$ satisfying $[\gamma] \in G \smallsetminus X$. Then we denote with $h_{\Omega} \in \Hom_{\bar{\calC}}(\bar{\bfP},\bar{U}^* \otimes \bar{U})$ the morphism
\[
 h_{\Omega} := \sum_{\mu \in \{ \gamma \} + \calH_r} \sum_{i=1}^{m_{\mu,\gamma}} \rmd^H(V_{\mu}) \cdot ((g_{\mu,\gamma,i})^* \otimes f_{\mu,\gamma,i}) \circ \bar{d}_{\mu,\gamma}.
\]

Next, we denote with $f_{\lambda \otimes 1} \in \Hom_{\bar{\calC}}(\bar{U},\bar{U}^* \otimes \bar{U})$ the unique morphism which sends the generator $1 \in \bar{U}$ to $\lambda \otimes 1 \in \bar{U}^* \otimes \bar{U}$, where $\lambda$ is the right integral of $\bar{U}_q \frakg$, and we consider a morphism $s_{\bar{\bfP}} \in \Hom_{\bar{\calC}}(\bar{\bfP},\bar{U} \otimes \bar{\bfP})$ satisfying $(\varepsilon \otimes \id_{\bar{\bfP}}) \circ s_{\bar{\bfP}} = \id_{\bar{\bfP}}$. Then we denote with $h_{\lambda} \in \Hom_{\bar{\calC}}(\bar{\bfP},\bar{U}^* \otimes \bar{U})$ the morphism
\[
 h_{\lambda} := \left( f_{\lambda \otimes 1} \otimes \bar{\bfepsilon} \right) \circ s_{\bar{\bfP}}.
\]

This allows us to review the recipe for the computation of the two invariants. If $(M,T,0,0)$ is a closed morphism of $\adCob_{\calC^H}^G$, if $e \subset T$ is a projective edge of color $V$, and if $L = L_1 \cup \ldots \cup L_{\ell} \subset S^3$ is a surgery presentation of $M$, then let us fix disjoint paths $\gamma_j \subset S^3 \smallsetminus (L \cup T)$ connecting $e$ to $L_j$ for every integer $1 \leq j \leq \ell$. Before starting, we perform special projective stabilizations both on $e \subset T$ and on $\bar{e} \subset \bar{T}$ at the intersection point with $\gamma_j$ for every integer $1 \leq j \leq \ell$, as shown in Figure \ref{F:projective_stabilizations}, 
\begin{figure}[t]
 \centering
 \includegraphics{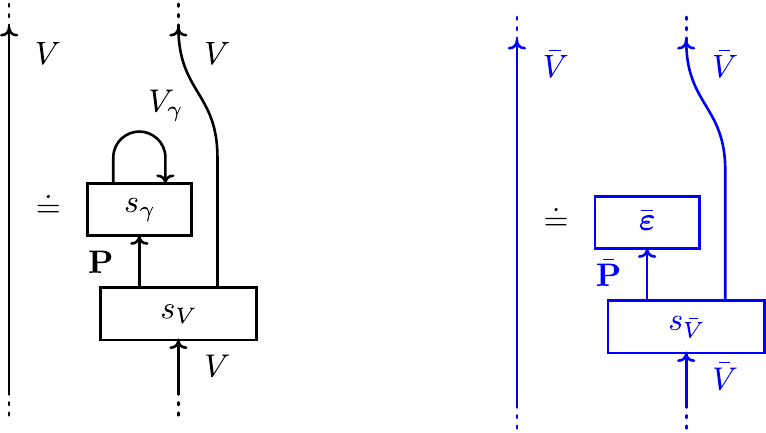}
 \caption{Projective stabilizations on $e \subset T$ and on $\bar{e} \subset \bar{T}$.}
 \label{F:projective_stabilizations}
\end{figure}
where $s_V \in \Hom_{\calC^H}(V,\bfP \otimes V)$ is a section of $\bfepsilon \otimes \id_V$, where $s_{\bar{V}} \in \Hom_{\bar{\calC}}(\bar{V},\bar{\bfP} \otimes \bar{V})$ is its image under $\Phi_{\calC}$, and where $s_{\gamma} \in \Hom_{\calC^H}(\bfP,W_{\gamma,\gamma})$ is a morphism satisfying $\rev_{V_{\gamma}} \circ s_{\gamma} = \bfepsilon$. Next, we isotope the $s_\gamma$-colored and the $\bar{\bfepsilon}$-colored coupons along the path $\gamma_j$ until the intersection point with $L_j$. This is our initial configuration.

Let us start from the CGP invariant. First, we need to slide every $V_\gamma$-colored edge along the corresponding component $L_j$, so to turn the surgery presentation $L$ into a computable one. Proposition 3.1 of \cite{D17} implies the choice of the weight $\gamma \in \ddot{\frakh}^*$ satisfying $[\gamma] \in G \smallsetminus X$ is inconsequential. This produces the $\calC^H$-colored ribbon graph represented in the top-left corner of Figure \ref{F:Omega-stabilization}. \begin{figure}[t]
 \centering
 \includegraphics{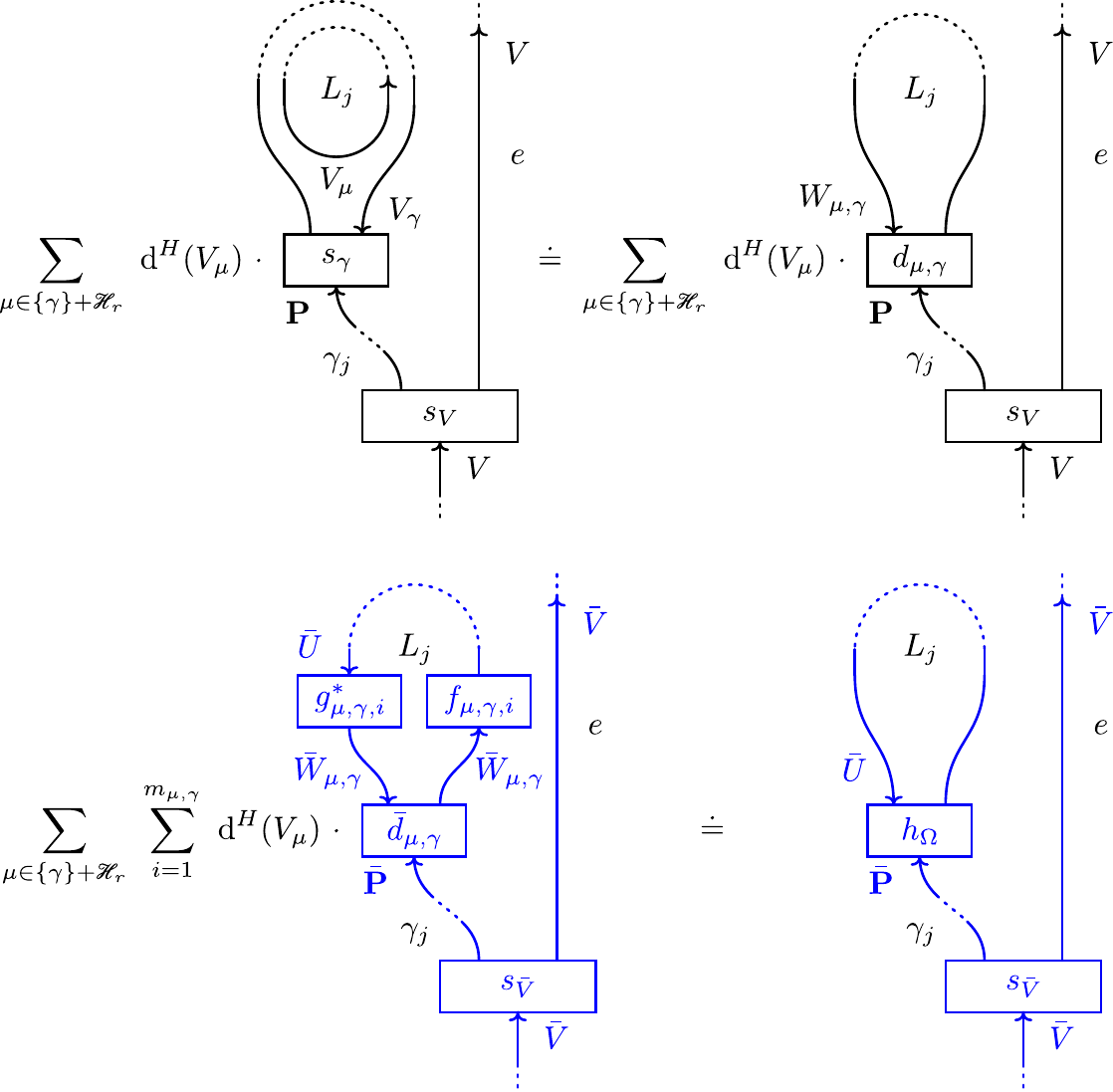}
 \caption{Skein equivalences of $\calC^H$-colored and $\bar{\calC}$-colored ribbon graphs defining $(L \cup \bar{T})_{h_{\Omega}}$.}
 \label{F:Omega-stabilization}
\end{figure} Up to skein equivalence of $\calC^H$-colored ribbon graphs, we can replace a tubular neighborhood of $L_j$ as shown in the top-right corner of Figure \ref{F:Omega-stabilization}, thus obtaining a $\smash{\calC^H_{[0]}}$-colored ribbon graph. This means we can apply the functor $\Phi_{\calR}$ which, up to skein equivalence of $\bar{\calC}$-colored ribbon graphs, produces the $\bar{\calC}$-colored ribbon graph represented in the bottom-left corner of Figure \ref{F:Omega-stabilization}. Again up to skein equivalence of $\bar{\calC}$-colored ribbon graphs, we can replace a tubular neighborhood of $L_j$ as shown in the bottom-right corner of Figure \ref{F:Omega-stabilization}. The resulting $\bar{\calC}$-colored ribbon graph is denoted $(L \cup \bar{T})_{h_{\Omega}}$, and is said to be obtained from the surgery presentation $L$ and from the admissible $\smash{\calC^H_{[0]}}$-colored ribbon graph $T$ by \textit{$\Omega$-stabilization along the paths $\gamma_1,\ldots,\gamma_{\ell}$}. By construction, using Lemma \ref{L:comparison_of_renormalized_graph_invariants}, we have 
\[
 \rmN_{\calC^H}(M,T,0,0) = \alpha \calD_{\Omega}^{-1-\ell} \delta_{\Omega}^{-\sigma(L)} F'_{\bar{\calC}} \left( (L \cup \bar{T})_{h_{\Omega}} \right).
\]

Let us move on to discuss the renormalized Hennings invariant. First, we need to interpret every component $L_j$ as a red edge, and to label it with the regular representation $\bar{U}$. This produces the $\bar{\calC}$-colored bichrome graph represented in the left-hand part of Figure \ref{F:lambda-stabilization}. 
\begin{figure}[t]
 \centering
 \includegraphics{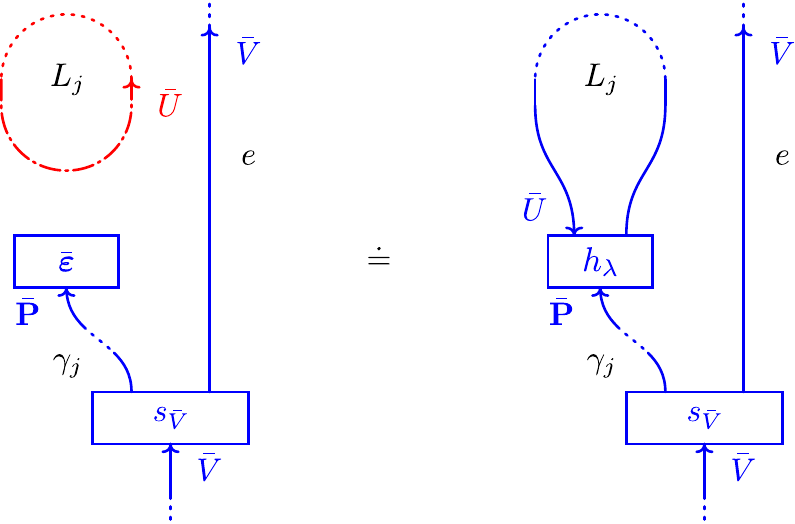}
 \caption{Skein equivalence of $\bar{\calC}$-colored bichrome graphs defining $(L \cup \bar{T})_{h_{\lambda}}$.}
 \label{F:lambda-stabilization}
\end{figure}
Up to skein equivalence of $\bar{\calC}$-colored bichrome graphs, we can turn every red component blue by replacing a tubular neighborhood of $L_j$ as shown in the right-hand side of Figure \ref{F:lambda-stabilization}. The resulting $\bar{\calC}$-colored ribbon graph is denoted $(L \cup \bar{T})_{h_{\lambda}}$, and is said to be obtained from the surgery presentation $L$ and from the admissible $\smash{\calC^H_{[0]}}$-colored ribbon graph $T$ by \textit{$\lambda$-stabilization along the paths $\gamma_1,\ldots,\gamma_{\ell}$}. By construction, using Lemma 3.8 of \cite{DGP18}, we have 
\[
 \rmH'_{\bar{\calC}}(M,\bar{T},0) = \calD_{\lambda}^{-1-\ell} \delta_{\lambda}^{-\sigma(L)} F'_{\bar{\calC}} \left( (L \cup \bar{T})_{h_{\lambda}} \right).
\]

Remark that, since $(L \cup \bar{T})_{h_{\Omega}}$ and $(L \cup \bar{T})_{h_{\lambda}}$ were obtained starting from surgery presentations computing $\rmN_{\calC^H}$ and $\rmH'_{\bar{\calC}}$ through operations which do not alter the values of $F'_{\calC^H}$ and of $F'_{\lambda}$ respectively, this means that both $F'_{\bar{\calC}} ( (L \cup \bar{T})_{h_{\Omega}})$ and $F'_{\bar{\calC}} ( (L \cup \bar{T})_{h_{\lambda}})$ are independent of the choice of the paths $\gamma_1, \ldots, \gamma_\ell$.

\subsection{Stabilization coefficients}

Next, we need to compare $\calD_{\Omega}$ with $\calD_{\lambda}$, $\delta_{\Omega}$ with $\delta_{\lambda}$, and $h_{\Omega}$ with $h_{\lambda}$. In order to do this, we will prove a key technical result. Let $\smash{(\bbS^1 \times \bbS^1)_{(-,\bar{\bfP})}}$ be the object of $\adCob_{\bar{\calC}}$ defined by
\[
 \smash{(\bbS^1 \times \bbS^1)_{(-,\bar{\bfP})}} := \left( S^1 \times S^1,P_{(-,\bar{\bfP})},\calL \right),
\]
where the blue $\bar{\calC}$-colored ribbon set $\smash{P_{(-,\bar{\bfP})}}$ is given by a single framed point with negative orientation and color $\bar{\bfP}$, and where the Lagrangian subspace $\calL$ is generated by the homology class of the curve $\{ (1, 0) \} \times S^1$. Analogously, let $\smash{\bbS^2_{((-,\bar{\bfP}),(-,\bar{\bfP}),(+,\bar{\bfP}))}}$ be the object of $\adCob_{\bar{\calC}}$ defined by
\[
 \smash{\bbS^2_{((-,\bar{\bfP}),(-,\bar{\bfP}),(+,\bar{\bfP}))}} := \left( S^2,P_{((-,\bar{\bfP}),(-,\bar{\bfP}),(+,\bar{\bfP}))},\{ 0 \} \right),
\]
where the blue $\bar{\calC}$-colored ribbon set $\smash{P_{((-,\bar{\bfP}),(-,\bar{\bfP}),(+,\bar{\bfP}))}}$ is given by three framed points, two with negative, one with positive orientation, and all with color $\bar{\bfP}$. Let us also consider the morphism $\smash{(\bbD^3 \smallsetminus \bbN^3)_{\bar{\bfP}} : (\bbS^1 \times \bbS^1)_{(-,\bar{\bfP})} \rightarrow \bbS^2_{((-,\bar{\bfP}),(-,\bar{\bfP}),(+,\bar{\bfP}))}}$ of $\adCob_{\bar{\calC}}$ defined by
\[
 \smash{(\bbD^3 \smallsetminus \bbN^3)_{\bar{\bfP}}} := \left( D^3 \smallsetminus N^3,T_{\bar{\bfP}},0 \right),
\]
where $N^3 \subset D^3$ is an open tubular neighborhood of the curve $\{ 0 \} \times \frac{1}{2} \cdot S^1 \subset D^3$, and where the $\bar{\calC}$-colored framed tangle $T_{\bar{\bfP}}$ is represented in Figure \ref{F:disc_minus_solid_torus}.

\begin{figure}[ht]
 \centering
 \includegraphics{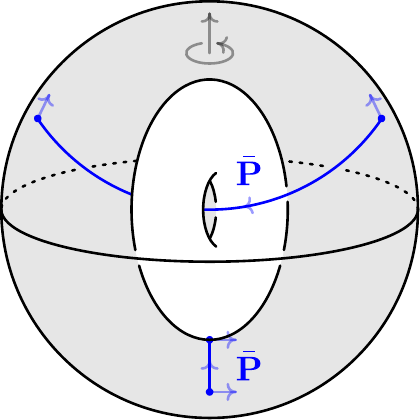}
 \caption{The morphism $(\bbD^3 \smallsetminus \bbN^3)_{\bar{\bfP}}$ of $\adCob_{\bar{\calC}}$.}
 \label{F:disc_minus_solid_torus}
\end{figure}

\begin{lemma}\label{L:injective_map_from_torus_to_disc}
 The linear map
 \[
  \smash{\rmV_{\bar{\calC}} \left( (\bbD^3 \smallsetminus \bbN^3)_{\bar{\bfP}} \right) : \rmV_{\bar{\calC}} \left( (\bbS^1 \times \bbS^1)_{(-,\bar{\bfP})} \right) \rightarrow \rmV_{\bar{\calC}} \left( \bbS^2_{((-,\bar{\bfP}),(-,\bar{\bfP}),(+,\bar{\bfP}))} \right)}
 \]
 is injective.
\end{lemma}

\begin{proof}
 As we will show, the proof follows rather directly from the surjectivity of
 \[
  \rmV'_{\bar{\calC}} \left( (\bbD^3 \smallsetminus \bbN^3)_{\bar{\bfP}} \right) : \rmV'_{\bar{\calC}} \left( \bbS^2_{((-,\bar{\bfP}),(-,\bar{\bfP}),(+,\bar{\bfP}))} \right) \rightarrow \rmV'_{\bar{\calC}} \left( (\bbS^1 \times \bbS^1)_{(-,\bar{\bfP})} \right).
 \]
 Indeed, a vector of the form
 \[
  \sum_{i=1}^m \alpha_i \cdot \left[ (\bbD^3 \smallsetminus \bbN^3)_{\bar{\bfP}} \circ \bbM_i \right] \in \rmV_{\bar{\calC}} \left( \bbS^2_{((-,\bar{\bfP}),(-,\bar{\bfP}),(+,\bar{\bfP}))} \right)
 \]
 for some $\alpha_1,\ldots,\alpha_m \in \C$ and some $[\bbM_1],\ldots,[\bbM_m] \in \rmV_{\bar{\calC}} \smash{\left( (\bbS^1 \times \bbS^1)_{(-,\bar{\bfP})} \right)}$ is trivial if and only if
 \[
  \sum_{i=1}^m \alpha_i \langle \bbM',(\bbD^3 \smallsetminus \bbN^3)_{\bar{\bfP}} \circ \bbM_i \rangle_{\bbS^2_{((-,\bar{\bfP}),(-,\bar{\bfP}),(+,\bar{\bfP}))}} = 0
 \]
 for every $[\bbM'] \in \rmV'_{\bar{\calC}} \smash{\left( \bbS^2_{((-,\bar{\bfP}),(-,\bar{\bfP}),(+,\bar{\bfP}))} \right)}$, where for every object $\bbSigma$ of $\adCob_{\bar{\calC}}$ the linear map
 \begin{gather*}
  \langle \cdot, \cdot \rangle_{\bbSigma} : \rmV'_{\bar{\calC}}(\bbSigma) \otimes \rmV_{\bar{\calC}}(\bbSigma) \rightarrow \C
 \end{gather*}
 denotes the non-degenerate pairing induced by the universal construction in Section 3.3 of \cite{DGP18}. Then, since
 \[
 \langle \bbM',(\bbD^3 \smallsetminus \bbN^3)_{\bar{\bfP}} \circ \bbM \rangle_{\bbS^2_{((-,\bar{\bfP}),(-,\bar{\bfP}),(+,\bar{\bfP}))}} = \langle \bbM' \circ (\bbD^3 \smallsetminus \bbN^3)_{\bar{\bfP}},\bbM \rangle_{(\bbS^1 \times \bbS^1)_{(-,\bar{\bfP})}}
 \]
 for every $[\bbM] \in \rmV_{\bar{\calC}} \smash{\left( (\bbS^1 \times \bbS^1)_{(-,\bar{\bfP})} \right)}$ and every $[\bbM'] \in \rmV'_{\bar{\calC}} \smash{\left( \bbS^2_{((-,\bar{\bfP}),(-,\bar{\bfP}),(+,\bar{\bfP}))} \right)}$, the injectivity of 
 \[
  \rmV_{\bar{\calC}} \left( (\bbD^3 \smallsetminus \bbN^3)_{\bar{\bfP}} \right) : \rmV_{\bar{\calC}} \left( (\bbS^1 \times \bbS^1)_{(-,\bar{\bfP})} \right) \rightarrow \rmV_{\bar{\calC}} \left( \bbS^2_{((-,\bar{\bfP}),(-,\bar{\bfP}),(+,\bar{\bfP}))} \right)
 \]
 is equivalent to the surjectivity of 
 \[
  \rmV'_{\bar{\calC}} \left( (\bbD^3 \smallsetminus \bbN^3)_{\bar{\bfP}} \right) : \rmV'_{\bar{\calC}} \left( \bbS^2_{((-,\bar{\bfP}),(-,\bar{\bfP}),(+,\bar{\bfP}))} \right) \rightarrow \rmV'_{\bar{\calC}} \left( (\bbS^1 \times \bbS^1)_{(-,\bar{\bfP})} \right).
 \]
 
 \begin{figure}[t]
  \centering
  \includegraphics{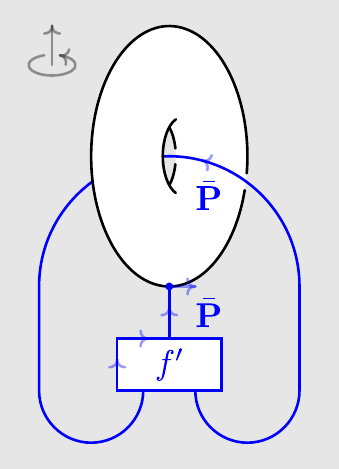}
  \caption{The $\bar{\calC}$-colored ribbon graph $T'_{f'} \subset (D^3 \smallsetminus N^3) \cup_{S^2} \bar{D^3}$.}
  \label{F:torus_prime_space}
 \end{figure}
 
 In order to prove that $\rmV'_{\bar{\calC}}((\bbD^3 \smallsetminus \bbN^3)_{\bar{\bfP}})$ is surjective we remark that, as soon as an object $\bbSigma = (\varSigma,P,\calL)$ of $\adCob_{\bar{\calC}}$ features a projective blue point of $P$ in every connected component of $\varSigma$, the proof of Proposition 3.13 of \cite{DGP18} can be repeated to show that the linear map 
 \[
  \begin{array}{rccc}
   \pi'_{\bbSigma} : & \calV'(M';\bbSigma) & \rightarrow & \rmV'_{\bar{\calC}}(\bbSigma) \\
   & T' & \mapsto & [M',T',0]
  \end{array}
 \]
 is surjective for every connected 3-dimensional cobordism $M'$ from $\Sigma$ to $\varnothing$. This means that every vector in $\rmV'_{\bar{\calC}}((\bbS^1 \times \bbS^1)_{(-,\bar{\bfP})})$ can be described by a linear combination of $\bar{\calC}$-colored bichrome graphs inside $(D^3 \smallsetminus N^3) \cup_{S^2} \bar{D^3}$ from $P_{(-,\bar{\bfP})}$ to $\varnothing$. Let $T'$ be such a $\bar{\calC}$-colored bichrome graph. Up to isotoping its $\bar{\bfP}$-colored edge intersecting $P_{(-,\bar{\bfP})}$, we can make sure there is a projective edge of $T'$ piercing the solid torus $N^3$. But now, thanks to Proposition \ref{P:barP_proj_gen}, $\bar{\bfP}$ is a projective generator of $\bar{\calC}$. This means that, up to skein equivalence, we can insert a pair of coupons joined by a $\bar{\bfP}$-colored edge. As a direct consequence, $T'$ is skein equivalent to a $\bar{\calC}$-colored ribbon graph like the one represented in Figure \ref{F:torus_prime_space} for some $f' \in \Hom_{\bar{\calC}}(\bar{\bfP} \otimes \bar{\bfP}^*,\bar{\bfP})$. Clearly every vector of this form lies in the image of $\rmV'_{\bar{\calC}}((\bbD^3 \smallsetminus \bbN^3)_{\bar{\bfP}})$. \qedhere
\end{proof}

Let us consider now the morphisms $(\bbS^1 \times \bbD^2)_{h_{\Omega}} : \varnothing \rightarrow (\bbS^1 \times \bbS^1)_{(-,\bar{\bfP})}$ and $(\bbS^1 \times \bbD^2)_{h_{\lambda}} : \varnothing \rightarrow (\bbS^1 \times \bbS^1)_{(-,\bar{\bfP})}$ of $\adCob_{\bar{\calC}}$ defined by
\begin{gather*}
 (\bbS^1 \times \bbD^2)_{h_{\Omega}} := (S^1 \times D^2,T_{h_{\Omega}},0), \\
 (\bbS^1 \times \bbD^2)_{h_{\lambda}} := (S^1 \times D^2,T_{h_{\lambda}},0)
\end{gather*}
where the $\bar{\calC}$-colored ribbon graphs $T_{h_{\Omega}}$ and $T_{h_{\lambda}}$ are represented in the left hand part and in the right hand part of Figure \ref{F:Omega_lambda_solid_torus} respectively.

\begin{figure}[ht]
 \centering
 \includegraphics{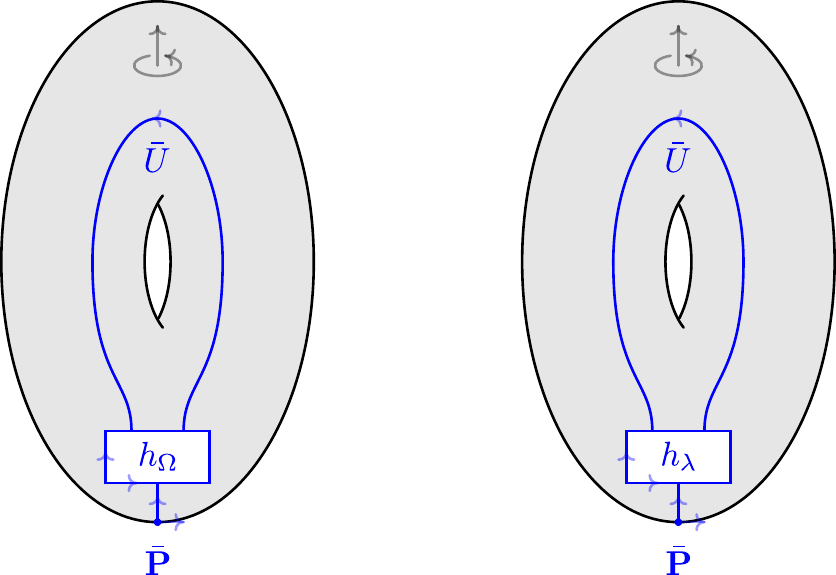}
 \caption{The $\bar{\calC}$-colored ribbon graphs $T_{h_{\Omega}},T_{h_{\lambda}} \subset S^1 \times D^2$.}
 \label{F:Omega_lambda_solid_torus}
\end{figure}

\begin{lemma}\label{L:cutting_inside_ball}
 The morphisms $(\bbS^1 \times \bbD^2)_{h_{\Omega}}$ and $(\bbS^1 \times \bbD^2)_{h_{\lambda}}$ of $\adCob_{\bar{\calC}}$ satisfy
 \[
  [(\bbS^1 \times \bbD^2)_{h_{\Omega}}] = \alpha \cdot [(\bbS^1 \times \bbD^2)_{h_{\lambda}}] \in \rmV_{\bar{\calC}} \smash{\left( (\bbS^1 \times \bbS^1)_{(-,\bar{\bfP})} \right)},
 \]
 where $\alpha$ is the coefficient introduced in Remark \ref{R:comparison_traces}. Furthermore, there exist compatible choices for the coefficients $\calD_{\Omega}$ and $\calD_{\lambda}$ yielding
 \[
  \calD_{\Omega} = \alpha \calD_{\lambda}, \quad \delta_{\Omega} = \delta_{\lambda}.
 \]
\end{lemma}

\begin{proof}
 We start by proving $[(\bbS^1 \times \bbD^2)_{h_{\Omega}}]$ and $[(\bbS^1 \times \bbD^2)_{h_{\lambda}}]$ are linearly dependent in $\rmV_{\bar{\calC}} \smash{\left( (\bbS^1 \times \bbS^1)_{(-,\bar{\bfP})} \right)}$. This is done by using Lemma \ref{L:injective_map_from_torus_to_disc}. Indeed, on one hand, Lemma \ref{L:cutting_bfP} gives the equality
 \[
  [(\bbD^3 \smallsetminus \bbN^3)_{\bar{\bfP}}) \circ (\bbS^1 \times \bbD^2)_{h_{\Omega}}] = r^{2N} \left| \calH_r \right| \cdot [(D^3,T_{D^3},0)] \in \rmV_{\bar{\calC}} \left( \bbS^2_{((-,\bar{\bfP}),(-,\bar{\bfP}),(+,\bar{\bfP}))} \right),
 \]
 where the $\bar{\calC}$-colored ribbon graph $T_{D^3}$ is represented in Figure \ref{F:cutting_inside_ball}.
 \begin{figure}[t]
  \centering
  \includegraphics{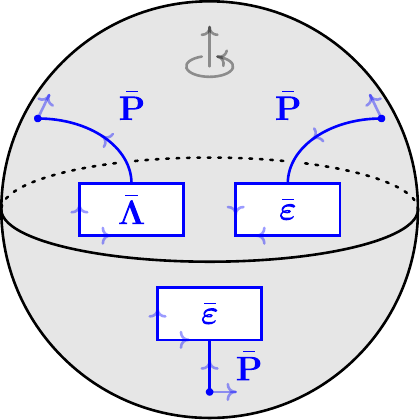}
  \caption{The $\bar{\calC}$-colored ribbon graph $T_{D^3} \subset D^3$.}
  \label{F:cutting_inside_ball}
 \end{figure}
 On the other hand, $\bar{U}$ is a projective generator of $\bar{\calC}$, which means
 \[
  \id_{\bar{\bfP}} = \sum_{i=1}^m g_{\bar{\bfP},i} \circ f_{\bar{\bfP},i}
 \]
 for some morphisms $f_{\bar{\bfP},i} \in \Hom_{\bar{\calC}}(\bar{\bfP},\bar{U})$ and $g_{\bar{\bfP},i} \in \Hom_{\bar{\calC}}(\bar{U},\bar{\bfP})$. Then, thanks to Lemma 3.6 of \cite{DGP18} combined with Lemma \ref{L:dimOne}, we know there exists a non-zero coefficient $\beta \in \C^*$ giving the skein equivalence of Figure \ref{F:cutting_red}.
 \begin{figure}[b]
  \centering
  \includegraphics{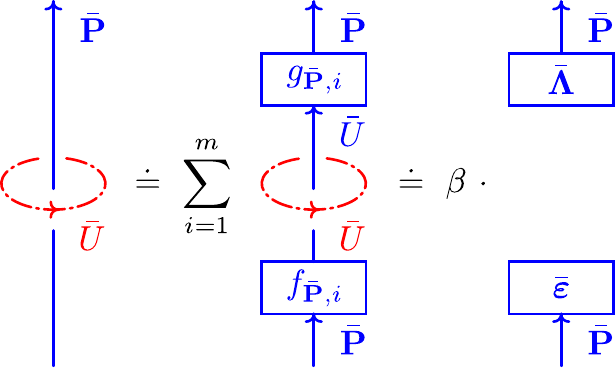}
  \caption{Cutting a $\bar{\bfP}$-colored edge with a red meridian.}
  \label{F:cutting_red}
 \end{figure}
 This gives
 \[
  [(\bbD^3 \smallsetminus \bbN^3)_{\bar{\bfP}}) \circ (\bbS^1 \times \bbD^2)_{h_{\lambda}}] = \beta \cdot [(D^3,T_{D^3},0)] \in \rmV_{\bar{\calC}} \left( \bbS^2_{((-,\bar{\bfP}),(-,\bar{\bfP}),(+,\bar{\bfP}))} \right).
 \]
 Therefore, we get
 \[
  \beta \cdot [(\bbS^1 \times \bbD^2)_{h_{\Omega}}] = r^{2N} \left| \calH_r \right| \cdot [(\bbS^1 \times \bbD^2)_{h_{\lambda}}].
 \] 
 
 Next, this relation allows us to compare the stabilization coefficients. Indeed, if $T_{\pm \Omega}$ and $T_{\pm \lambda}$ denote the $\bar{\calC}$-colored ribbon graphs represented in Figure \ref{F:comparison_stabilization_coefficients},
 \begin{figure}[t]
  \centering
  \includegraphics{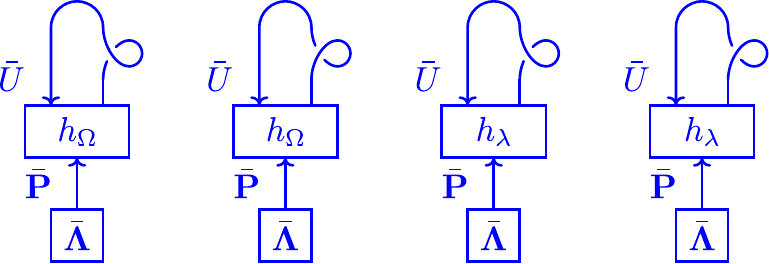}
  \caption{The $\bar{\calC}$-colored ribbon graphs $T_{-\Omega},T_{+\Omega},T_{-\lambda},T_{+\lambda} \subset S^3$.}
  \label{F:comparison_stabilization_coefficients}
 \end{figure}
 then, by definition of $h_\Omega$, we have
 \[
  F'_{\bar{\calC}}(T_{\pm \Omega}) = \Delta_{\pm \Omega} \bar{\rmt}_{\bar{\bfP}}(\bar{\bfLambda} \circ \bar{\bfepsilon}) = \alpha \Delta_{\pm \Omega} \rmt^H_{\bfP}(\bfLambda \circ \bfepsilon) = \alpha \Delta_{\pm \Omega},
 \]
 and analogously, by definition of $h_\lambda$, we have
 \[
  F'_{\bar{\calC}}(T_{\pm \lambda}) = \Delta_{\pm \lambda} \bar{\rmt}_{\bar{\bfP}}(\bar{\bfLambda} \circ \bar{\bfepsilon}) = \alpha \Delta_{\pm \lambda} \rmt^H_{\bfP}(\bfLambda \circ \bfepsilon) = \alpha \Delta_{\pm \lambda}.
 \]
 This means
 \[
  \rmH'_{\bar{\calC}}(S^3,T_{\pm \Omega},0) = \alpha \calD_{\lambda}^{-1} \Delta_{\pm \Omega}, \quad \rmH'_{\bar{\calC}}(S^3,T_{\pm \lambda},0) = \alpha \calD_{\lambda}^{-1} \Delta_{\pm \lambda}.
 \]
 But now, thanks to the previous equality, we have 
 \[
  \beta \rmH'_{\bar{\calC}}(S^3,T_{\pm \Omega},0) = r^{2N} \left| \calH_r \right| \rmH'_{\bar{\calC}}(S^3,T_{\pm \lambda},0).
 \]
 This gives
 \[
  \beta \Delta_{\pm \Omega} = r^{2N} \left| \calH_r \right| \Delta_{\pm \lambda}.
 \]
 In particular, combining this equality with the explicit value of $\Delta_{-\Omega} \Delta_{+\Omega}$ given by Figure \ref{F:cutting_ev-coev}, we can choose
 \[
  \calD_{\Omega} = r^N \sqrt{\left| \calH_r \right|}, \quad
  \calD_{\lambda} = \frac{\beta}{r^N \sqrt{\left| \calH_r \right|}}.
 \]
 This immediately implies 
 \[
  \delta_{\Omega} = \frac{\calD_{\Omega}}{\Delta_{- \Omega}} = \frac{\calD_{\lambda}}{\Delta_{- \lambda}} = \delta_{\lambda}.
 \]

 Furthermore, we can now compute the equality
 \[  
  \beta = \frac{r^{2N} \left| \calH_r \right|}{\alpha}.
 \]
 Indeed, let us consider the object $\bbS^2_{(+,\bar{\bfP})}$ of $\adCob_{\bar{\calC}}$ defined by $(S^2,P_{(+,\bar{\bfP})},\{ 0 \})$, where the blue $\bar{\calC}$-colored ribbon set $P_{(+,\bfP)}$ is given by a single framed point with positive orientation and color $\bar{\bfP}$, and let us consider the closed morphism $\bbS^2_{(+,\bar{\bfP})} \times \bbS^1$ of $\adCob_{\bar{\calC}}$ defined by $(S^2 \times S^1,P_{(+,\bar{\bfP})},0)$. The strategy will be to compute the renormalized Hennings invariant of $\bbS^2_{(+,\bar{\bfP})} \times \bbS^1$ in two different ways. On one hand, the isomorphism $\rmV_{\bar{\calC}} \smash{\left( \bbS^2_{(+,\bar{\bfP})} \right)} \cong \Hom_{\bar{\calC}}(\one,\bar{\bfP})$ gives
 \[
  \rmH'_{\bar{\calC}} \left( \bbS^2_{(+,\bar{\bfP})} \times \bbS^1 \right) = \dim_{\C} \left( \rmV_{\bar{\calC}} \left( \bbS^2_{(+,\bar{\bfP})} \right) \right) = \dim_{\C} \left( \Hom_{\bar{\calC}}(\one,\bar{\bfP}) \right) = 1,
 \]
 where the last equality follows from Lemma \ref{L:dimOne}. On the other hand, we can choose a surgery presentation of $S^2 \times S^1$ composed of a single unknot of framing 0. This choice determines the $\bar{\calC}$-colored bichrome graph $H_{\bar{\bfP}}$ given by a positive Hopf link of framing 0, with one red component colored with $\bar{U}$ and one blue component colored with $\bar{\bfP}$. Therefore, we get
 \begin{align*}
  \rmH'_{\bar{\calC}} \left( \bbS^2_{(+,\bar{\bfP})} \times \bbS^1 \right) 
  &= \calD_{\lambda}^{-2} F'_{\bar{\calC}}(H_{\bar{\bfP}}) \\
  &= \frac{r^{2N} \left| \calH_r \right| \bar{\rmt}_{\bar{\bfP}}(\bar{\bfLambda} \circ \bar{\bfepsilon})}{\beta} \\
  &= \frac{r^{2N} \left| \calH_r \right| \rmt^H_{\bfP}(\bfLambda \circ \bfepsilon)}{\alpha \beta} \\
  &= \frac{r^{2N} \left| \calH_r \right|}{\alpha \beta}. \qedhere
 \end{align*}
\end{proof}

\subsection{Proof of Theorem \ref{T:main_result}}

We are now ready to prove Theorem \ref{T:main_result}. As explained in the proof of Lemma \ref{L:cutting_inside_ball}, using the explicit value of $\Delta_{-\Omega} \Delta_{+\Omega}$ given by Figure \ref{F:cutting_ev-coev}, we can choose the square roots $\calD_{\Omega}$ and $\calD_{\lambda}$ to be of the form
\[
 \calD_{\Omega} = r^N \sqrt{\left| \calH_r \right|}, \quad
 \calD_{\lambda} = \frac{r^N \sqrt{\left| \calH_r \right|}}{\alpha},
\] 
Also, let us set
\[
 \delta := \delta_{\Omega} = \delta_{\lambda}.
\]

\begin{proof}[Proof of Theorem \ref{T:main_result}]
 If $M$ is a closed 3-manifold, and $T \subset M$ is an admissible $\smash{\calC^H_{[0]}}$-colored ribbon graph, then let $e \subset T$ be a projective edge of color $V$, let $L = L_1 \cup \ldots \cup L_{\ell} \subset S^3$ be a surgery presentation of $M$, let $\gamma_j \subset S^3 \smallsetminus (L \cup T)$ be disjoint paths connecting $e$ to $L_j$ for every integer $1 \leq j \leq \ell$, and let $(L \cup \bar{T})_{h_{\Omega}}$ and $(L \cup \bar{T})_{h_{\lambda}}$ be $\bar{\calC}$-colored ribbon graphs obtained by $\Omega$-stabilization and by $\lambda$-stabilization along $\gamma_1, \ldots, \gamma_{\ell}$, as explained in Subsection \ref{Subs:stabilized_surgery_presentations}. Then we have
 \begin{align*}
  \rmN_{\calC^H}(M,T,0,0) &= \alpha \calD_{\Omega}^{-1-\ell} \delta^{-\sigma(L)} F'_{\bar{\calC}}((L \cup \bar{T})_{h_{\Omega}}) \\
  &= \alpha^{-\ell} \calD_{\lambda}^{-1-\ell} \delta^{-\sigma(L)} F'_{\bar{\calC}}((L \cup \bar{T})_{h_{\Omega}}) \\
  &= \alpha^{-\ell} \calD_{\lambda}^{-\ell} \delta^{-\sigma(L)} \rmH'_{\bar{\calC}}(S^3,(L \cup \bar{T})_{h_{\Omega}},0) \\
  &= \calD_{\lambda}^{-\ell} \delta^{-\sigma(L)} \rmH'_{\bar{\calC}}(S^3,(L \cup \bar{T})_{h_{\lambda}},0) \\
  &= \calD_{\lambda}^{-1-\ell} \delta^{-\sigma(L)} F'_{\bar{\calC}}((L \cup \bar{T})_{h_{\lambda}}) \\
  &= \rmH'_{\bar{\calC}}(M,\bar{T},0),
 \end{align*}
 where the second and the fourth equalities follow from Lemma \ref{L:cutting_inside_ball}.
\end{proof}

\appendix

\section{Quantum groups}\label{A:quantum_groups}

In this appendix we collect some standard definitions related to quantum groups, see \cite{CP95,J96,KS97,M11} for more details. Let $\frakg$ be a simple complex Lie algebra of rank $n$ and dimension $2N + n$, let $B$ be its Killing form, let $\frakh$ be a Cartan subalgebra of $\frakg$, let $\Phi$ be the corresponding root system, let $\Phi_+$ be a choice of a set of positive roots of $\frakg$, and let $\{ \alpha_1, \ldots, \alpha_n \}$ be an ordering of its set of simple roots. Let $A = ( a_{ij} )_{1 \leqslant i,j \leqslant n}$ be the corresponding Cartan matrix, which is the integral matrix given by
\[
 a_{ij} := \frac{2 B^*(\alpha_i,\alpha_j)}{B^*(\alpha_i,\alpha_i)},
\]
where $B^*$ is the symmetric bilinear form on $\frakh^*$ determined by the restriction of $B$ to $\frakh$ under the isomorphism which identifies a vector $H \in \frakh$ with the linear form $B(H,\cdot) \in \frakh^*$, and let $\{ H_1,\ldots, H_n \}$ be the basis of $\frakh$ determined by $\alpha_j (H_i) = a_{ij}$ for all integers $1 \leqslant i,j \leqslant n$. For every $\alpha \in \Phi_+$ we set
\[
 d_{\alpha} := \frac{B^*(\alpha,\alpha)}{\min \{ B^*(\alpha_i, \alpha_i) \mid 1 \leqslant i \leqslant n \}}
\]
and for every integer $1 \leqslant i \leqslant n$ we use the short notation $d_i := d_{\alpha_i}$. We denote with $\langle \cdot,\cdot \rangle$ the symmetric bilinear form on $\frakh^*$ determined by $\langle \alpha_i,\alpha_j \rangle = d_i a_{ij}$ for all integers $1 \leqslant i,j \leqslant n$, and we denote with $\omega_1,\ldots,\omega_n$ the corresponding fundamental dominant weights, which are the vectors of $\frakh^*$ determined by the condition $\langle \omega_i,\alpha_j \rangle = d_i \delta_{ij}$ for every $i,j = 1,\ldots,n$. We denote with $\Lambda_R$ the root lattice, which is the subgroup of $\frakh^*$ generated by simple roots, and we denote with $\Lambda_W$ the weight lattice, which is the subgroup of $\frakh^*$ generated by fundamental dominant weights. If $q$ is a formal parameter, then for every $\alpha \in \Phi_+$ we set $q_{\alpha} := q^{d_{\alpha}}$, for all $k \geqslant \ell \in \N$ we define
\begin{gather*}
 \{ k \}_{\alpha} := q_{\alpha}^k - q_{\alpha}^{-k}, \quad [k]_{\alpha} := \frac{\{ k \}_{\alpha}}{\{ 1 \}_{\alpha}},
 \quad [k]_{\alpha}! := [k]_{\alpha}[k-1]_{\alpha}\cdots[1]_{\alpha}, \\
 \sqbinom{k}{\ell}_{\alpha} := \frac{[k]_{\alpha}!}{[\ell]_{\alpha}![k-\ell]_{\alpha}!},
\end{gather*}
and for every integer $1 \leqslant i \leqslant n$ we use the short notation 
\[
 q_i := q_{\alpha_i}, \quad \{ k \}_i := \{ k \}_{\alpha_i}, \quad [k]_i := [k]_{\alpha_i},
 \quad [k]_i! := [k]_{\alpha_i}!, \quad
 \sqbinom{k}{\ell}_i := \sqbinom{k}{\ell}_{\alpha_i}.
\]
Let $\calU_q \frakg$ denote the \textit{quantum group of $\frakg$}, which is the $\C(q)$-algebra with generators
\[
 \{ K_i,K_i^{-1},E_i,F_i \mid 1 \leqslant i \leqslant n \}
\]
and relations
\begin{gather*}
 K_i K_i^{-1} = K_i^{-1} K_i = 1, \quad [K_i,K_j] = 0, \\
 K_i E_{j} K_i^{-1} = q_i^{a_{ij}} \cdot E_{j}, \quad
 K_i F_{j} K_i^{-1} = q_i^{- a_{ij}} \cdot F_{j}, \\
 [E_i,F_j] = \delta_{ij} \cdot \frac{K_i - K_i^{-1}}{q_i-q_i^{-1}}
\end{gather*}
for all integers $1 \leqslant i,j \leqslant n$ and
\begin{gather*}
 \displaystyle \sum_{k=0}^{1 - a_{ij}} (-1)^{k} \sqbinom{1-a_{ij}}{k}_i \cdot E_{i}^{k} E_{j} E_{i}^{1-a_{ij}-k} = 0, \\
 \displaystyle \sum_{k=0}^{1 - a_{ij}} (-1)^{k} \sqbinom{1-a_{ij}}{k}_i \cdot F_{i}^{k} F_{j} F_{i}^{1-a_{ij}-k} = 0\phantom{,}
\end{gather*}
for all integers $1 \leqslant i,j \leqslant n$ with $i \neq j$. Then $\calU_q \frakg$ can be made into a Hopf algebra by setting
\begin{align*}
 \Delta(K_i) &= K_i \otimes K_i, & \varepsilon(K_i) &= 1, & S(K_i) &= K_i^{- 1}, \\
 \Delta(E_i) &= E_i \otimes K_i + 1 \otimes E_i, & \varepsilon(E_i) &= 0, & S(E_i) &= -E_i K_i^{-1}, \\
 \Delta(F_i) &= F_i \otimes 1 + K_i^{-1} \otimes F_i, & \varepsilon(F_i) &= 0, & S(F_i) &= - K_i F_i
\end{align*}
for all integers $1 \leqslant i \leqslant n$. For every
\[
 \mu = \sum_{i=1}^n m_i \cdot \alpha_i \in \Lambda_R
\]
we use the notation
\[
 K_{\mu} := \prod_{i=1}^n K_i^{m_i},
\]
and for every $\alpha \in \Phi_+$ we define root vectors $E_{\alpha}$ and $F_{\alpha}$ as follows: first, we consider the Weyl group $W$ of $\frakg$ associated with $\frakh$, which is the subgroup of $\GL(\frakh^*)$ generated by reflections
\[
 \begin{array}{rccc}
  s_i : & \frakh^* & \rightarrow & \frakh^* \\
  & \alpha_j & \mapsto & \alpha_j - a_{ij} \cdot \alpha_i
 \end{array}
\]
for every integer $1 \leqslant i \leqslant n$. Next, we consider the unique element $w_0 \in W$ corresponding to a word of maximal length in the generators. The choice of a decomposition $w_0 = s_{i_1} \circ \cdots \circ s_{i_N}$ determines a total order on the set of positive roots
\[
 \Phi_+ = \{ \alpha_{i_1},s_{i_1}(\alpha_{i_2}),\ldots,(s_{i_1} \circ \cdots \circ s_{i_{N-1}})(\alpha_{i_N}) \}. 
\]
Then, for every integer $1 \leqslant i \leqslant n$, we consider the automorphism $T_i$ of $\calU_q \frakg$ determined by 
\begin{align*}
 T_i(K_j) &:= K_j K_i^{-a_{ij}}, \\
 T_i(E_j) &:= 
 \begin{cases}
  - F_iK_i \phantom{- K_i^{-1}E_i} \hspace{150pt} & i = j, \\
  \displaystyle \sum_{k=0}^{-a_{ij}} (-1)^{k} \frac{q_i^{a_{ij} + k}}{[k]_i![-a_{ij}-k]_i!} \cdot E_i^{k} E_j E_i^{-a_{ij}-k} & i \neq j,
 \end{cases} \\
 T_i(F_j) &:= 
 \begin{cases}
  - K_i^{-1}E_i \phantom{- F_iK_i} \hspace{150pt} & i = j, \\
  \displaystyle \sum_{k=0}^{-a_{ij}} (-1)^{-a_{ij}-k} \frac{q_i^{k}}{[k]_i![-a_{ij}-k]_i!} \cdot F_i^{k} F_j F_i^{-a_{ij}-k} & i \neq j.
 \end{cases}
\end{align*}
Now, for every integer $1 \leqslant k \leqslant N$, we set 
\[
 \beta_k := (s_{i_1} \circ \cdots \circ s_{i_{k-1}})(\alpha_{i_k}) \in \Phi_+
\]
and
\[
 E_{\beta_k} := (T_{i_1} \circ \cdots \circ T_{i_{k-1}})(E_{i_k}) \in \calU_q \frakg, \quad 
 F_{\beta_k} := (T_{i_1} \circ \cdots \circ T_{i_{k-1}})(F_{i_k}) \in \calU_q \frakg.
\]

 

\end{document}